\documentclass[10pt,twocolumn]{IEEEtran}

\usepackage[dvips]{graphicx}
\usepackage{import}
\usepackage{epsfig}
\usepackage[colorlinks=true]{hyperref}
\usepackage{pstricks}
\usepackage{amsmath,amssymb,amsfonts,amsthm}
\usepackage{color}

\renewcommand{\c}{\mathfrak{c}}

\newcommand{\lift}[1]{\hat{#1}}

\newcommand{\G}{\mathcal{G}}
\newcommand{\V}{\mathcal{V}}
\newcommand{\E}{\mathcal{E}}
\newcommand{\X}{\mathcal{X}}
\renewcommand{\S}{\mathcal{S}}

\newcommand{\siggy}{_}

\newcommand{\RR}{\mathbb{R}}

\newcommand{\PP}{\mathbb{P}}

\newcommand{\sfM}{\mathsf{M}}
\newcommand{\sfI}{\mathsf{I}}
\newcommand{\sfR}{\mathsf{R}}
\newcommand{\sfE}{\mathsf{E}}
\newcommand{\sfS}{\mathsf{S}}

\newcommand{\sfm}{\mathsf{m}}
\newcommand{\sfi}{\mathsf{i}}
\newcommand{\sfr}{\mathsf{r}}
\newcommand{\sfe}{\mathsf{e}}
\newcommand{\sfs}{\mathsf{s}}

\newtheorem{lemma}{Lemma}
\newtheorem{defin}{Definition}
\newtheorem{prob}{Problem}
\newtheorem{cor}{Corollary}
\newtheorem{thm}{Theorem}

\newcommand{\demi}{\tfrac{1}{2}}

\newcommand{\marg}{p}
\newcommand{\dist}{x}
\newcommand{\adj}{\dag}

\begin{document}
\definecolor{lightgray}{gray}{0.6}

\title{Lifting Markov Chains To Mix Faster:\\ Limits and Opportunities}
\author{Simon Apers,  Francesco Ticozzi and Alain Sarlette
\thanks{Work partially by the University of Padua under the projects CPDA140897/14 and QCON.}
\thanks{S. Apers is with the Department of Electronics and Information Systems, Ghent University, Technologiepark 914, 9052 Zwijnaarde-Ghent, Belgium ({\tt\small simon.apers@ugent.be}); A. Sarlette is with Ghent University and with the QUANTIC lab, INRIA Paris, rue Simone Iff 2, 75012 Paris, France ({\tt\small alain.sarlette@inria.fr}); F. Ticozzi is with the
Dipartimento di Ingegneria dell'Informazione, Universit\`a di
Padova, via Gradenigo 6/B, 35131 Padova, Italy ({\tt\small ticozzi@dei.unipd.it}), and the Department of Physics and Astronomy, Dartmouth College, 3127 Wilder, Hanover, NH (USA). 
}
}

\maketitle

\begin{abstract}
Lifted Markov chains are Markov chains on graphs with added local ``memory'' and can be used to mix towards a target distribution faster than their memoryless counterparts. Upper and lower bounds on the achievable performance have been provided {\em under specific assumptions}. In this paper, we analyze which assumptions and constraints are relevant for mixing, and how changing these assumptions affects those bounds. Explicitly, we find that requesting mixing on either the original or the full lifted graph, and allowing for reducible lifted chains or not, have no essential influence on mixing time bounds. On the other hand, allowing for {\em suitable initialization} of the lifted dynamics and/or {\em imposing invariance of the target distribution for any initialization} do significantly affect the convergence performance. The achievable convergence speed for a lifted chain goes from diameter-time to no acceleration over a standard Markov chain, with conductance bounds limiting the effectiveness of the intermediate cases. In addition, we show that the relevance of imposing ergodic flows depends on the other criteria. The presented analysis allows us to clarify in which scenarios designing lifted dynamics can lead to better mixing, and provide a flexible framework to compare lifted walks with other acceleration methods.
\end{abstract}




\section{Introduction}

The importance of algorithms based on Markov chains has been widely demonstrated. In computer science, random walks and Markov chain Monte Carlo form the backbone of many randomized algorithms to solve tasks such as approximating the volume of convex bodies \cite{dyer1991} or the permanent of a non-negative matrix \cite{jerrum2004}, or to solve combinatorial optimization problems using simulated annealing methods \cite{kirkpatrick1983}. In physics, Markov chain Monte Carlo is an indispensable tool for sampling and simulation of many-body systems. Some examples are the use of Glauber dynamics to simulate the Ising model \cite{martinelli1999} or the Metropolis-Hastings algorithm \cite{metropolis1953,hastings1970} to sample from the Gibbs distribution.

In general, a Markov chain can be used to sample from a probability distribution that is not directly available nor fully known. This is done by engineering a stochastic evolution, obeying the imposed constraints and using only the available information, that converges, or {\em mixes}, to an equilibrium which coincides with the target distribution. What is critical to these applications is {\em how fast} the process approaches its equilibrium. Estimating its convergence speed or \textit{mixing time} is often a main hurdle \cite{aldous2002,levin2009}. Approaches range from estimating the spectral radius of the transition map \cite{lawler1988,diaconis1991} to the use of advanced coupling and stopping time arguments \cite{lovasz1998}.

In order to accelerate the convergence speed, other stochastic evolutions beyond Markov chains have been proposed: among these are {\em lifted Markov chains} \cite{diaconis2000,chen1999,diaconis2013}, where a memory effect is introduced by considering Markov evolutions on an enlarged state space, and {\em quantum walks} \cite{kempe2003}, where the evolution mechanisms are allowed to exploit quantum superposition effects.
In this work, we focus on the former, and discuss {\em under which constraints} the non-Markovian effects induced by the lifting allow for an improvement over the best available Markov chain on the original state space. More specifically, we consider lifts of a Markov chain originally specified on a graph, and we establish how lower bounds on the mixing time stem not only from the topology of the graph itself, but also from a combination of properties imposed on the lifted evolution. Our results add to the recently revived interest in irreversible and beyond-Metropolis-Hastings techniques, some interesting results of which are presented in \cite{turitsyn2011,rey2016,ramanan2016,bierkens2016}. 
In addition, one of the motivations for the present work is to pave the way towards a quantitative and fair comparison between lifted chains and quantum walks: although very different from an internal dynamics perspective, both these models rely on non-Markovian effects and share many formal similarities that can be made precise within the proposed framework. 

Before comparing our analysis with the existing literature, we first introduce the models of interest in Section \ref{sec:lifts}, and the constraints we consider in Section \ref{sec:classes}. This will allow us to properly collocate our results with respect to existing ones and highlight the importance of two aspects that have a key role in lift designs: the ability of locally {\em initializing} the lifted chain, at least as a function of the initial condition on the reference graph, and the constraints regarding {\em invariance} of the target distribution. These two properties allow us to immediately identify some ``extreme'' scenarios, described in Section \ref{sec:extreme}, where either lifting does not yield any advantage or it allows for reaching the target in the minimum possible {\em finite} time, corresponding to the diameter of the graph. This is achieved by embedding in the lift a set of {\em stochastic bridges} \cite{pavon2010,georgiou2015}.

Between these two extreme scenarios, we establish lower bounds on the mixing time when only one of the above constraints is requested, together with constraints on the reducibility of the lifted Markov chain, its own mixing properties and its ergodic flows, in Section \ref{sec:conductance}.
These bounds depend on the {\em conductance} of the graph, which provides a richer description of the graph topology than the diameter. 
In particular, we show that a conductance bound for the mixing time of lifted Markov chains holds under either of two seemingly unconnected constraints: (i) if we impose that the lifted chain mixes from any initial state in the entire lifted state space (this result is essentially due to \cite{chen1999}); or (ii) we impose that the lifted dynamics leaves the target distribution invariant. Conductance bounds are typically stricter than the diameter-time bound, yet examples show how they still allow for the lifted chains to significantly outperform the best possible standard Markov chain.
We note that only one of the two constraints can be imposed if one wants accelerated convergence: when both (i) and (ii) are imposed, the results of the previous sections show that no speedup is possible as compared to a simple Markov chain. On the other hand, if both are relaxed, we are exactly in the diameter-time mixing scenario mentioned above.

A discussion of the results and an outlook on future developments is provided in Section \ref{sec:conclusion}. While some of the specific results we include here may not be quite surprising or could even seem to have incremental value, we believe that as a result of a comprehensive analysis, covering the potential performance of lifted chains for most reasonable design scenarios, we clearly highlight a point that has been to some extent overlooked in previous work: the key role of design constraints in developing mixing algorithms. In this sense, even if the constructions in our proofs are more existential than practical, we believe that the present paper fills a hole in the relevant literature and might inspire new developments towards more practical algorithms.


\section{Mixing Dynamics on Graphs and their Lifts}\label{sec:lifts}

Consider a graph $\G=(\V,\E)$, with $\V$ a set of $N$ nodes, which we label as $\V=\{i=1,\ldots,N\}$ and $\E$ the set of edges i.e. pairs of nodes. Throughout the paper, graphs are assumed to be connected, and functions on the node space will be represented as vectors in $\RR^N$. More details on the linear algebraic representation, in particular on product spaces, are provided in Appendix \ref{app:linear}.  Let $\PP\siggy{N}$ be the set of probability vectors in $\RR^N,$ i.e.~each $\marg \in \PP\siggy{N}$ satisfies $x_i \geq 0$ $\forall i=1,2,...,N$ and $\sum_{i \in \V} \marg_i = 1$. Each such $\marg_i$ represents the probability of node $i\in\V$.
The basic task, throughout the paper, is the following: 
\begin{prob}[Design of mixing dynamics] Design a discrete-time stochastic dynamical system that converges towards a target probability distribution $\pi$ on $\V$, as fast as possible from any initial condition on $\V$, while respecting the locality associated to the graph.
\end{prob}
By respecting the locality of the graph, we mean that the evolution over one time step can only involve transitions between nodes connected by an edge.
The target distributions of interest in the following are $\pi$ with full support, namely such that $\pi_i>0$ for all $i \in [1,\ldots,N].$ These are a dense set in $\PP\siggy{N},$ and allow us to simplify the analysis. 
A number of further constraints on the solution process, which precisely specify the design problem at hand, will be discussed in Section \ref{sec:classes} and, as we shall see, the performance of the best available solution will critically depend on these assumptions.

A common approach to address the mixing problem is to converge towards the steady-state distribution $\pi \in \PP\siggy{N}$ by iterating a linear, stochastic discrete-time map
$$\marg (t+1) = P \, \marg (t) \; .$$
The transition map $P$ contains the probabilities $P_{j,i}$ to jump from node $i$ to node $j$ and it must satisfy the {\em locality constraints} induced by $\G$: namely, $P_{j,i} = 0$ if $\E$ contains no edge from $i$ to $j$. The goal is to converge towards $\pi$ as fast as possible from any initial $\marg (0)$. Note that, as $\pi_i>0$ for all $i,$ $P$ must be irreducible in order to allow $\marg(t)$ to converge to $\pi$ from any $p(0)$. $P$ defines a Markov chain on the node set $\V$.

It has been shown \cite{diaconis2000,chen1999} that convergence can be accelerated, under the same locality constraints, by adding memory to the Markov chain. Formally, this leads to Markov chains on lifted graphs or, for short, {\em lifted Markov chains}.
\begin{defin} \label{def:lift}
A graph $\lift{\G}=(\lift{\V},\lift{\E})$ on $\lift{N}$ nodes is called a {\em lift of $\G$} if there exists a surjective map $\c: \lift{\V} \mapsto \V,$ such that: 
\[(i,j) \in \lift{\E}  \; \text{ implies } \; (\c(i),\c(j)) \in \E.\]  
We denote by $\c^{-1}$ the map that takes as input a single node $k \in \V$ and outputs the set of nodes $j \in \lift{\V}$ for which $\c(j)=k$. 
\end{defin}
We will denote by $\dist\in \PP\siggy{\lift{N}}$ a distribution over the lifted graph nodes $\lift{\V}$. The associated marginal distribution over $\V$ is given by
$\; \marg_k = \sum_{j \in \c^{-1}(k)} \; \dist_j.$
In vector representation, this induces the linear map
\begin{equation}\label{eq:SS1}
\marg = C \dist \, ,
\end{equation}
with $C$ a matrix of zeroes and ones.
In a {\em lifted Markov chain} for $\G$, the distribution $\marg(t)$ on $\V$ at time $t$ is obtained as the marginal of $\dist(t),$ whose evolution is generated by a linear, stochastic discrete-time map \emph{on the lifted graph}:
\begin{equation}\label{eq:SS2}
\dist(t+1) = A \, \dist(t)
\end{equation}
where $A$ satisfies the locality constraints on $\lift{\G}$ induced by the underlying $\G$, i.e.~$A_{j,\ell}\neq 0$ only if $(\c(j),\c(\ell))$ is an edge of ${\cal G}.$  

The locality of the lifted chain equivalently means that for each $\dist$ there exists a stochastic matrix $P^{(x)}$ satisfying the locality constraints of $\G$ and such that \eqref{eq:SS2} corresponds to $\; \marg(t+1) = P^{(\dist)}\, \marg(t),$ with $\marg(t)=C\dist(t)$ as in \eqref{eq:SS1}. Explicitly, this is given by $\;P^{(\dist)} = C\, A \, B^{(\dist)}$ where $B^{(\dist)}$ is a linear stochastic map from $\mathbb{R}^N$ to $\mathbb{R}^{\lift{N}}$, with
\[\; B^{(\dist)}_{i,j} = \frac{\dist_i(t)}{\sum_{k \in \c^{-1}(j)} \dist_k(t)} \;  \textrm{ if } \c(i)=j, \;\] and $B^{(x)}_{i,j} = 0 \;$ otherwise.\vspace{2mm}

\noindent{\em Remark:} 
In Markov modeling, $A$ would be called a hidden Markov chain \cite{rabiner1986}. In algorithmic applications, which motivate our setting, the pair $(\lift{\G},A)$ is to be \emph{designed} in order to accelerate the convergence towards $\pi$ with respect to the (best) Markov chain $P$ on the original graph $\G$. The use of $x$ as the variable in the larger state space is inspired by system theory, and it also hints to the fact that the dynamics of $x$ are to be designed.




\section{Mixing time and design scenarios}\label{sec:classes}


The main message for this paper is that what is possible in terms of mixing performance is critically dependent on some constraints and insensitive to some others, besides the locality associated to the graph of interest $\G$. Hence, before we proceed, we need to carefully specify the constraints and emerging design scenarios. It will be clear that even the most natural definition of mixing time crucially depends on the imposed constraints. 

\subsection{Initialization of the lift}
When a stochastic dynamics is seen as an algorithm, one must specify how to initialize it.
We assume that the input of such an algorithm is a node of $\V$, chosen accordingly to an initial distribution $\marg(0)=\marg_0,$ on which we have no control. We consider two possible scenarios and the associated set of initial distributions $\S$:

\begin{itemize}
\item[$(\sfS)$] In a first scenario, it is possible to initialize the 
lifted evolution depending on the initial input.
The algorithm design can then, in addition to $\lift{\G}$ and $A$, choose how to lift the weight $\marg_k(0)$ attributed to each node $k \in \V$ of the original graph $\G$, onto a distribution over its associated lifted nodes $\dist_{\c^{-1}(k)}(0)$ in agreement with the locality constraints. We further require, in order to maintain linearity of the whole process, that the designed initialization is a \emph{linear} map $F : \marg(0) \mapsto \dist(0),$ satisfying:
$$C F \marg(0) = \marg(0) \quad \forall \marg(0) \; .$$
The latter plays the role of a {\em locality constraint} on $F,$ as it  implies $F_{k,j} = 0$ whenever $\c(k) \neq j$.
In this scenario, the set of relevant initial conditions $\S$ for the {\em lifted} Markov chain does not comprise all possible distributions $\dist(0)$ on $\lift\V$, but only those of the form $\dist(0) = F \marg(0)$, for all initial distributions $\marg(0)$ on $\V$. 

\item[$(\sfs)$] In some other cases, there might be no control over the initialization of the lifted dynamics. The set of relevant initial conditions $\S$ is then the whole $\PP\siggy{\lift{N}}$.
\end{itemize}

\subsection{Invariance of the target marginal}

For a Markov chain on $\G,$ mixing is necessarily towards its unique \emph{invariant} distribution, i.e.~$P \pi = \pi$. For a lifted Markov chain, however, even if the marginal converges to $\pi,$ in the transient $C\dist(t) = \pi$ does not necessarily imply $C\dist(t+1)= \pi$. One may request that the system does not leave the target $\pi$ when it starts there at $t=0$. This mimics the invariance of the target in an un-lifted Markov chain. It arguably imposes to ``avoid unnecessary work'', and it plays an essential role towards interfacing the Markov chain with other algorithmic elements, in particular implementing the key task of \textit{amplification}, i.e., boosting the success probability of a randomized algorithm (in our case the closeness to the stationary distribution) by rerunning the algorithm on its own output, see e.g.~\cite{motwani2010}. We thus identify two possible scenarios:
\begin{itemize}
\item[$(\sfi)$] We impose $C \dist(t) = \pi$ for all $t > 0$ whenever $C \dist(0) = \pi$, for all allowed $\dist(0)\in\S$.

\item[$(\sfI)$] We allow $C\dist(t) \neq \pi$ for some $t\geq 0$ and some $\dist(0)$ even when $C\dist(0)=\pi$.
\end{itemize}

\subsection{Marginal vs lift mixing time}
A time-homogenous Markov chain on $\G$ associated to a transition matrix $P$ is said to mix to $\pi$ if $P\pi=\pi$ 
{and for all $\epsilon>0$ there exists $\tau(\epsilon)>0$ such that, for all $p \in \PP\siggy{N}$, we have: 
  	  	\[ \|P^t p-\pi\|_{TV} \leq \epsilon \; 
  	  						\text{ for all } \; t \geq \tau(\epsilon). \]}
We call $\tau(\epsilon)$ its $\epsilon$-mixing time\footnote{For completeness, we recall that the total variation distance between two distributions $\marg(1)$ and $\marg(2)$ is $1/2$ times the 1-norm of their difference $\Vert \marg(1) - \marg(2)\Vert_1 = \sum_{i=1}^N |\marg_i(1)-\marg_i(2)|$.}.
\noindent It is typical to consider $\tau(1/4)$ as a reference mixing time. 
Most papers that bound the mixing time of lifted Markov chains, analyze in fact how fast the state \emph{on the lifted space} $\dist$ converges to its stationary value $\bar{\dist}$, i.e.~they consider $\|A^t \dist -\bar{\dist}\|_{TV}$. Our original algorithmic task however is to accelerate convergence of the \emph{marginal} $\marg(t)=C\dist(t)$, compared to the performance of the original chain $P$. To this aim, we define the marginal mixing time.
\begin{defin}[Marginal mixing time]
A lifted chain on $\lift\G$ associated to a transition matrix $A$ is said to mix to the marginal $\pi$ on $\G$ from initial conditions $\S$, if
{for all $\epsilon>0$ there exists $\tau_M(\epsilon)>0$ such that for all $\dist \in \S$ we have: 
  	  	\[ \|CA^t \dist-\pi\|_{TV} \leq \epsilon \; 
  	  						\text{ for all } \; t \geq \tau_M(\epsilon). \]}
We call $\tau_M(\epsilon)$ its {\em $\epsilon$-marginal mixing time}.
\end{defin}

{Of course, $\tau_M(\epsilon) \leq \tau(\epsilon)$ for all $\epsilon$. While the convergence of $\dist$ is indeed a sufficient proxy for the convergence of $\marg=C\dist$, it is not truly necessary. One might argue that for generic $A$ we expect $\dist$ and $\marg=C\dist$ to have similar convergence speeds. However in our application, the specific lifts designed to speed up convergence are all but generic, and this distinction could become relevant.} For instance, some typical designs involve constructions where the lifted Markov chain $\dist$ in fact does not converge to a stationary value, but the projected state $\marg$ does -- see the periodic clock lift in Appendix \ref{ssec:tools}.
Furthermore, it is easy to construct lifted walks where $\marg$ converges much faster than $\dist$.
One should wonder whether conversely, a Markov chain with quickly converging $\marg$ can always be adapted to also have quickly converging $\dist$, or whether sometimes there is a strict advantage to be gained when we are only interested in $\marg$.
We therefore need to specify which type of convergence we are requesting:
\begin{itemize}
\item[$(\sfM)$]  convergence of the marginal $\marg(t)$ towards $\pi$, as measured by $\tau_M(\epsilon).$
\item [$(\sfm)$] convergence of $\dist(t)$ towards $\bar{\dist}$, as measured by $\tau(\epsilon).$
\end{itemize}

\subsection{Reducibility of the lift}
A Markov chain $P$ that globally converges to a unique stationary distribution $\pi$ with $\pi_i > 0$ $\forall i$, must be irreducible -- that is, there cannot exist a partition of $\V$ into subsets $\X$ and $\V \setminus \X$ such that $P_{i,j}=0$ for all $(i,j)$ with $i \in \X$ and $j \in \V \setminus \X$. However, the same need not necessarily apply to a lifted Markov chain $A$, depending on the rest of the setting. Hence two scenarios emerge:
\begin{itemize}
\item[$(\sfR)$] The lifted Markov chain $A$ is allowed to be reducible. 
\item [$(\sfr)$] The lifted chain $A$ must be irreducible.
\end{itemize}
It is easy to come up with combinations of the other constraints for which $(\sfR)$  makes sense -- as well as cases where we have convergence of any $\dist(0)$ towards a unique $\lift{\pi}$, but where $\lift{\pi}_i = 0$ for some $i$.

\subsection{Matching ergodic flows}

When $\G$ and $\pi$ are given, one can think about the optimization problem of computing the compatible Markov chain $P$ with fastest convergence towards $\pi$, see for instance \cite{boyd2004} for symmetric $P$. For the lifted Markov chain, one would perform a similar optimization on $A$. Yet in some cases, one may be given a reference $P$ whose associated ``typical flows'', for some reason, should not be altered or increased by the transitions induced by the lifted chain \cite{chen1999}. In that case, one would request the optimization on $A$ to take into account the reference $P$ on $\G$.

More precisely, for a given Markov chain $P$, the associated ergodic flows are defined by $Q^{(P)}_{i,j} = P_{i,j}\pi_j$ i.e.~the weight that flows from $j$ to $i$ when the system is on the steady state $y=\pi$. For a lifted chain $A$ and (one of) its steady state(s) $\lift{\pi}$, one can similarly define ergodic flows $\lift{Q}^{(A;\lift{\pi})}_{i,j} = A_{i,j} \lift{\pi}_j$. To compare ``typical flows'' in $A$ and $P$, one then compares $Q^{(P)}_{i,j}$ and \[\lift{Q}^{(A;\lift{\pi})}_{\c^{-1}(i),\c^{-1}(j)} = \sum_{\ell,k : \c(\ell)=i,\c(k)=j}  \lift{Q}^{(A;\lift{\pi})}_{\ell,k}.\] The Markov chain $\tilde{P}^{(A;\lift{\pi})}$ on $\G$ defined by $\tilde{P}^{(A;\lift{\pi})}_{i,j} \pi_j = \lift{Q}^{(A;\lift{\pi})}_{\c^{-1}(i),\c^{-1}(j)}$ is called the {\em induced chain on $\G$ by the lift $A$ and distribution $\lift{\pi}$}. This leads to the following scenarios.
\begin{itemize}
\item[$(\sfe)$] A reference Markov chain $P$ is given and the ergodic flows of $A$ must match the ergodic flows of $P$, or in other words, the induced Markov chain $\tilde{P}^{(A;\lift{\pi})}$ must be equal to the reference $P$. When $A$ has several steady states $\lift{\pi}$ and hence several ergodic flows $\lift{Q}^{(A;\lift{\pi})}$, they must all have the same induced chain $\tilde{P}^{(A;\lift{\pi})} = P$.

\item[$(\sfE)$] No constraint is imposed on the ergodic flows of the lifted Markov chain.
\end{itemize}

It is worth noting that the definition of ergodic flows for irreducible $A$ is in fact an extension of the traditional definition, see for instance \cite{aldous2002}.


\subsection{On combinations of constraints and relations between lift design problems}

In the following, in order to compactly refer to a set of requirements in the statements of our results and to carry out proper comparisons, we shall specify an alternative (upper- or lower-case letter) for each of the scenarios described in the previous subsections, e.g. $(\sfs\sfI\sfm\sfr\sfE)$. 

We denote {\em the set of dynamics allowed by design scenario by the corresponding string of 5 letters}.
 A shorter string may be used to indicate properties that hold for all the compatible alternatives: e.g. $(\sfs\sfI\sfm\sfE)$ includes all scenarios in $(\sfs\sfI\sfm\sfr\sfE)$ and $(\sfs\sfI\sfm\sfR\sfE)$.
The lower- and upper-case letters in the previous subsections have been chosen so that the upper-case properties yield scenarios that are {\em less constrained} than the one with the same rest of the string but a lower-case letter. This implies that a scenario associated to multiple capital letters always includes all lifts available in a scenario where some of those letters are substituted by their lowercase versions, i.e.:
{\begin{eqnarray}\label{eq:properties1}
(\sfS)\vert_x \supseteq (\sfs)\vert_x & ; & 
(\sfI)\vert_x \supseteq (\sfi)\vert_x \\
\nonumber
(\sfM)\vert_x \supseteq (\sfm)\vert_x & ; & 
(\sfR)\vert_x \supseteq (\sfr)\vert_x \;\; ; \; (\sfE)\vert_x \supseteq (\sfe)\vert_x  \; ,
\end{eqnarray}
where $\vert_x$ denotes arbitrary} but equal choices on the left and right hand sides for the other four letters.
From this observation follows that lower or at most equal optimal mixing times are expected when relaxing the constraints from lowercase to uppercase.  This will be useful in assessing the potential speedup of the scenarios. For example, scenarios with $\sfS$, $\sfI$, $\sfM$, $\sfR$ have the possibility to better improve mixing performance, compared to the fastest non-lifted Markov chain $P,$ with respect to their non-capital counterparts. 

The comparison of $\sfE,\sfe$ scenarios and non-lifted chains needs particular care, because $\sfe$ is the only constraint type that uses a reference non-lifted Markov chain $P.$ Indeed, requiring no particular ergodic flows allows to select better lifted walks, 
but it simultaneously gives the freedom to optimize $P$ for the non-lifted process.

Considering a binary alternative for each property yields $2^5 = 32$ possible (fully specified) scenarios in total, which are not necessarily all equally relevant. One might question for instance the practical situations where one would encounter $(\sfS\sfm)$ or $(\sfm\sfR)$. We leave such considerations to the end-user, and here report the bounds on the mixing speed for all scenarios.
Finally, a very important aspect of engineered mixing dynamics is the amount of resources they require for implementation. Following other abstract work on mixing bounds, we will not consider this aspect in detail, as it further depends on the specific application and the available information. All the constructions we employ in the proofs, while having mostly existential rather than practical value, have a number of lifted nodes polynomial in $N$.

\subsection{Discussion of previous work}

Previous work by \cite{diaconis2000,diaconis2013} and \cite{chen1999} considered the $(\sfs\sfI\sfm\sfr)$ or $(\sfs\sfI\sfm\sfr\sfe)$ context. In particular, \cite{chen1999} established the conductance bound we derive in Theorem \ref{cor:conductancecor1}, for the particular scenario $(\sfs\sfI\sfm\sfr\sfe)$. More recently, this result was extended to the continuous-time case \cite{ramanan2016}, which we will not treat here. In the following sections, however, we will prove that such bounds can be derived from just the $(\sfs)$ and $(\sfs\sfe)$ constraints, and that these constraints are also in some sense necessary when other requirements are missing: if they are omitted then the diameter becomes the trivial yet tight lower bound.

Jung et al.~\cite{jung2010} introduced the concept of a ``pseudo-lift'' as a relaxation of the lifted Markov chain formalism, by allowing post-selection on the sampling outcome. They show that pseudo-lifts can mix in diameter time on any graph, yet due to post-selection their construction is effectively a Las Vegas--type algorithm. As our analysis is concerned with Monte Carlo--type algorithms, discussion of pseudo-lifts goes beyond the scope of the present work. It might however be possible to cast their construction as a higher-lifted Markov chain Monte Carlo algorithm. Our results then show that necessarily this construction would fail both the $(\sfs)$ and $(\sfi)$ constraints.

Theorem \ref{thm:diameter} on diameter-time mixing is reminiscent of finite-time consensus results such as those in \cite{hendrickx2014}. Here we have a more constrained setting, including positivity constraints and time-invariant dynamics. 
Allowing for a lifted dynamics takes care of the latter, as we shall see, while positivity is intrinsically built in our framework apparently without affecting the fast convergence.



\section{{Minimal and maximal acceleration of mixing}:\\ the interplay of invariance and initialization}\label{sec:extreme}

We start by identifying the scenarios for which the lift cannot provide any advantage in mixing time with respect to a regular Markov chain, and those that allow for the fastest (diameter time) mixing. Remarkably, the only constraints that are relevant to determine these ``extreme'' behaviors concern the capability of initializing the lift, together with the requirement on the invariance of the target distribution $\pi$.


\subsection{Scenarios where lifting does not speed up mixing}


We start by showing that, under the constraints $(\sfs\sfi),$ the lifted Markov chain cannot go faster than the best non-lifted chain $P$ compatible with the graph, {\em even if we only look at the marginal mixing time.} This is made precise in the following result.
\begin{thm}\label{thm:nolift} In all scenarios featuring constraints $(\sfs\sfi),$ for any lifted Markov chain $(\hat {\cal G},A)$ whose marginal $\marg_t=CA\dist_t$ mixes to $\pi,$ there exists a stochastic matrix $P^\mathfrak{q}$ such that $\marg_{t+1}=P^\mathfrak{q}\marg_t$ for all $t.$
\end{thm}

\begin{proof}
The essential idea of the proof is that, in order for the underlying lifted dynamics to satisfy invariance of $\pi$ for all initializations on $\lift{\V}$, it is necessary that any two $\dist^{(1)},\; \dist^{(2)}$ for which $C \dist^{(1)} = C \dist^{(2)}$, induce the same flow on $\G$.
Since we have $(\sfs)$, the lift can start from any distribution $x$ over $\lift{\V}$. Invariance $(\sfi)$ then requires that for any $x$ for which $Cx = \pi$, we have $C A x = \pi$.

Given a lifted Markov chain satisfying $(\sfs\sfi)$, consider a map $\mathfrak{q} : \V \mapsto \lift{\V}$ that maps every $j \in \V$ to a single node $k_j \in \lift{\V}$ for which $\c(k_j) = j$. Let $\dist=q(\marg)$ denote the distribution with $\dist _{\mathfrak{q}(j)} = \marg _j$ for all $j \in \V$, and $\dist _i=0$ for all remaining $i \in \lift{\V}$. Defining 
$$P^\mathfrak{q}_{i,j} = \sum_{\ell \in \c^{-1}(i)} A_{\ell,\mathfrak{q}(j)}\; ,$$
we will show that for any $\dist(t)$ with $\marg (t) = C \dist(t)$, the lifted Markov chain satisfies 
\begin{equation}\label{eqh1:tp}
\marg (t+1) = P^\mathfrak{q} \marg (t) \; ,
\end{equation}
i.e.~it behaves like the non-lifted Markov chain $P^\mathfrak{q}$.
Proving \eqref{eqh1:tp} amounts to proving that
\begin{equation}\label{eq:Pqdist}
C A \dist = P^\mathfrak{q}  C \dist
\end{equation}
for all $\dist \in \PP\siggy{\lift{N}}$. For any $\dist$ of the form $\dist = q(\marg)$, with $\marg \in  \PP\siggy{N}$, we indeed have \eqref{eq:Pqdist} by construction. For any other $\dist$, defining $\dist^{(q)} = q (C \dist)$, there remains to show that $C A \dist = C A \dist^{(q)}$.
To do so, select some $a>0$ such that $a \pi_j > \marg _j$ for all $j \in \V$ and define $\pi' = \eta\, (a \pi - y)$, with $1/\eta = \sum_{j \in \V} (a \pi_j - \marg _j) = a-1$ i.e.~$a=1+1/\eta$. It is easy to check that $\pi' \in \PP\siggy{N}$. Now select any distribution $x'$ over $\lift{\V}$ such that $C x' = \pi'$ and let $\dist^{(1)}= (x + x'/\eta) / a$, $\;\; \dist^{(2)}= (\dist^{(q)} + x'/\eta) / a,$\;\; which are properly normalized distributions. We then have by construction $a (\dist^{(1)}-\dist^{(2)}) = x-\dist^{(q)}$, and with $C \dist^{(1)} = C \dist^{(2)} = \pi$. Invariance $(\sfi)$ requires that $C A \dist^{(1)}=C A \dist^{(2)} = \pi$, which readily implies 
$\; C A (x-\dist^{(q)}) = 0\;$.
\end{proof}


\subsection{Scenarios where lifting allows for diameter-time mixing}

A basic bound on mixing time is that, under locality constraints, the equilibrium distribution cannot be reached in a time that is shorter than the graph diameter. The same bound holds also for time-inhomogeneous (non-lifted) dynamics, and it is easy to see that lifted dynamics must satisfy it as well. 
We next show that a class of scenarios allow for mixing in diameter time.
Remarkably, this is possible for any graph, as soon as we are allowed to suitably initialize the lifted chain and we don't necessarily require invariance of $y(0)=\pi$. The proof uses basic building blocks for constructing lifts, which we develop in Appendix \ref{ssec:tools}, and we urge the reader interested in the details to read this Appendix first. While scenarios including $(\sfr\sfe)$ constraints are not directly covered by the following result, we will provide a corollary that addresses a slightly relaxed problem.

\begin{thm}\label{thm:diameter}
All scenarios in $(\sfS\sfI),$ with the exception of $(\sfS\sfI\sfr\sfe),$ admit a lifted Markov chain for which $\tau_M(1/4) \leq \tau(1/4) \leq D_\G+1,$ with $D_\G$ the graph diameter; the associated lifted graph has of order $D_\G N^2$ nodes.
\end{thm}

\begin{proof}
We use the notation $e_i$ for a unit vector with all components 0 except $i$. Given any $\G$ and $\pi$, for each node $i \in \V$ we can build the stochastic bridge (see Appendix \ref{ssec:tools}) from a $\marg (0)$ concentrated on $i,$ towards the target $\marg(D_\G)=\pi$. We can then combine these bridges via a node-clock-lift (see Appendix \ref{ssec:tools}) into a single lifted Markov chain which, when initialized ($\sfS$) with $F_{(s=0,v_0=i,v=i),i} = 1$ $\forall i \in \V$ and all other $F_{i,j}=0$, converges exactly ($\epsilon =0$) to $\pi$ in $D_\G$ time steps (and stays there during the following steps). Moreover, any such $\dist(0)$ converges to $\dist(t\geq D_\G+1)=\dist(D_\G+1)=e_{D_\G+1} \otimes \pi \otimes \pi$. 
This proves the claim for $(\sfS\sfI\sfm\sfR\sfE)$ and $(\sfS\sfI\sfM\sfR\sfE)$, since the resulting lift is reducible and does not follow any specific ergodic flows. 
The latter is easy to circumvent with the singular definition of ergodic flows with $(\sfR)$: in the last term of \eqref{eq:node-clock}, replace $I_\V\otimes I_{\V}$ by $I_\V\otimes P$. Indeed, since $P \pi = \pi$, this will not change the final marginal, but it does ensure that the ergodic flows are exactly those of $P$ on $\pi$, proving the theorem for $(\sfS\sfI\sfR)$.
\vspace{1mm}

The case $(\sfS\sfI\sfr\sfE)$ is the last that remains to be covered: to this aim, we must do three things.
First, we add to the lift a small probability to jump back to the layer of initialization nodes, i.e.~we now replace the last term $e_{T+1} e_{T+1}^\adj \otimes I_{\V} \otimes I_{\V}$ in \eqref{eq:node-clock} by
$$ (1-\gamma) e_{T+1} e_{T+1}^\adj \otimes I_\V \otimes P + \gamma J $$
with $J_{(t=0,v,v),(t=D_\G+1,v_0,v)}=1$ for all $v_0,v \in \V$ and all other $J_{i,j}=0$, and $\gamma\ll 1$ a small positive parameter. Second, we drop the lifted nodes that are never populated in the full node-clock-lift, in order to obtain an irreducible graph. 
Third, we must adjust some transition values from $(1-\gamma) e_{T+1} e_{T+1}^\adj \otimes I_\V \otimes P$ to $(1-\gamma) e_{T+1} e_{T+1}^\adj\otimes I_\V \otimes \tilde{P}$ in order to ensure that $C\lift{\pi}=\pi$ the target steady state. To compute this adjustment, we first note that with $\tilde{P}$ the steady state $\lift{\pi}$ will be of the form
\begin{eqnarray}\label{eq:herepi}
\lift{\pi} &= \frac{1}{1+(D_\G+1)\gamma} e_{D_\G+1} \otimes \pi \otimes \tilde{\pi} \\
	&+ \sum_{t=0}^{D_\G} \, \frac{\gamma}{1 + (D_\G+1) \gamma} A^t F \tilde{\pi} \; ,
\end{eqnarray}
where $\tilde{\pi} = \gamma \pi + (1-\gamma)\tilde{P}\tilde{\pi}$. We will {\em a)} first compute a $\tilde{\pi}$ close to $\pi$ and such that $C \lift{\pi} = \pi$ in the expression \eqref{eq:herepi}, without caring about its relation to $\tilde{P}$; {\em b)} next, we will show how to construct a corresponding $\tilde{P}$.

\noindent $a)$ Let $ \sum_{t=0}^{D_\G}\; C A^t F =: (D_\G+1)B$. We thus want that $\left(\frac{1}{1+(D_\G+1)\gamma} \; + \; \frac{(D_\G+1) \gamma}{1 + (D_\G+1) \gamma} B \right)\; \tilde{\pi} = \pi$. By taking $\gamma$ small enough, we can ensure that this equation is evidently invertible and gives a solution $\tilde{\pi}$ that equals $\pi$ up to terms of order $\gamma$. In particular, we can make sure that $\tilde{\pi}_k \geq \alpha > 0$ for all $k$ and for some $\alpha$, for all $\gamma>0$ smaller than some $\gamma_0$.

\noindent $b)$ There remains to find $\tilde{P}$ such that $\tilde{P}\tilde{\pi} = \tfrac{1}{1-\gamma}(\tilde{\pi}-\gamma\pi)$. Since $P$ is irreducible, there must exist some $\beta>0$ such that the edges $(i,j)$ of $\G$ for which $P_{i,j}\geq \beta$, contain a rooted spanning tree $\G_\beta$. Writing $\tilde{P} = P + P'$ and defining $y = (P+\gamma/(1-\gamma)) \, (\tilde{\pi}-\pi)$, we have to solve
$$
P' \tilde{\pi} = y \;\; , \quad \sum_{k=1}^N\,P'_{k,\ell} = 0 \;\; \text{for } \ell=1,2,...,N \;$$ $$ \quad P'_{k,\ell} = 0 \text{ for all } (k,\ell) \notin \E \; ,
$$
knowing that $\sum_{k=1}^N y_k = 0$. Note that we impose no positivity constraint on $P+P'$ because, with $y$ of order $\gamma$, if we can solve the above system without singularities, then the $P'$ will always be of order $\gamma$ as well, and hence it will not disturb $P$ too much whenever $\gamma \ll \beta$. We can construct $P'$ from the rooted spanning tree $\G_\beta$ of $\G$. Consider any leaf $j$ of this spanning tree and its parent $k$. We set $P'_{j,k}=y_j/\tilde{\pi}_k$ to satisfy $P' \tilde{\pi} = y$ for row $j$. Furthermore, we add $-y_j/\tilde{\pi}_k$ to $P'_{k,k}$ to maintain $\sum_{\ell=1}^N\, P'_{\ell,k} = 0$. Once all the ``children'' $\{j\}$ of a node $k$ have been treated, we can turn to satisfying $P' \tilde{\pi} = y$ for row $k$ by setting the value of $P'_{k,\ell}$, associated to the parent $\ell$ of $k$, taking into account the contributions already present via $P'_{k,k} = \sum_j (-y_j/\tilde{\pi}_k)$. We thus set $P'_{k,\ell} = (y_k + \sum_j y_j)\,/\, \tilde{\pi}_\ell$, and again to maintain $\sum_{b=1}^N\, P'_{b,\ell} = 0$ we also add $- (y_k + \sum_j y_j)\,/\, \tilde{\pi}_\ell$ to $P'_{\ell,\ell}$. We can pursue this construction, with elements of $P'$ bounded by some multiple of $1/\tilde{\pi}_i < 1/\alpha$, until we reach the root. At this point, $P' \tilde{\pi} = y$ is satisfied for all rows except one, corresponding to the root $r$. Writing $\sum_{k \neq r} [P' \tilde{\pi}]_k - y_k = 0$ and from our problem setting $\sum_{k=1}^N\, P'_{k,\ell} \tilde{\pi}_\ell = \sum_{k=1}^N y_k = 0$, we see that in fact $P' \tilde{\pi} = y$ is satisfied for row $r$ as well, and we have constructed an appropriate $P'$, i.e.~$\tilde{P}$ close to $P$ such that $C\lift{\pi} = \pi$.
Thanks to the very low value of $\gamma$, a total variation distance $\leq 1/4$ towards $\pi$ is still reached as soon as a walk initialized with $\dist(0) = F \marg(0)$ reaches the top layer of our node-clock-lift, i.e.~after $D_\G+1$ steps, although the walk has not perfectly reached the steady state yet.
%
\end{proof}

While we do not have a full construction for the $(\sfS\sfI\sfr\sfe)$ case, we can use a variation of the construction for $(\sfS\sfI\sfr\sfE)$ provided in the above proof in order to approximately match some prescribed ergodic flows, with {\em arbitrary accuracy.} Formally, let us define the scenario $(\sfS\sfI\sfr\sfe_\delta)$, where $\sfe_\delta$ means that ergodic flows are allowed to deviate by no more than $\delta$ from the imposed ones, for any $\delta>0.$ We then have the following:

\begin{cor} \label{cor:diameter-time}
For $(\sfS\sfI\sfr\sfe_\delta),$ we can construct a lifted Markov chain with the mixing time $\tau_M(1/4) \leq \tau(1/4) \leq D_\G+1$.
\end{cor}
\proof
The construction described in the proof of Theorem \ref{thm:diameter} produces a $\tilde{P}$ close to $P$ and a distribution $\tilde{\pi}$ close to $\pi$ on the top layer of the node-clock-lift, such that the corresponding ergodic flows can be made $\delta$-close to those imposed by $P,\pi$, for any $\delta > 0$, by taking $\gamma$ small enough.
\qed\vspace{2mm}

{\em Remark:} It may be interesting to note that fast convergence of $\dist(t)$ towards $\lift{\pi}$, holds not only for the $1/4$ distance, but in fact up to any distance $\epsilon > 0$.
This can be seen by observing that the ``clock'' degrees of freedom (see Appendix \ref{ssec:tools}) undergo an independent Markov chain on $\V^{(clock)} = \{ s=0,1,...,D_\G+1 \}$ with transition matrix
\begin{equation}\label{eq:forTh3Pprime}
P^{(clock)} =  \sum_{i=0}^{D_\G} e_{i+1}e_i^\adj + (1-\gamma) e_{D_\G+1} e_{D_\G+1}^\adj + \gamma e_0 e_{D_\G+1}^\adj \; .
\end{equation}
This is a particular Markov chain on a path or cycle graph, which is independent of the original problem and $\G$, except for its length $D_\G$.
Then write $y^{(clock)} = \pi^{(clock)} + q$ with $q$ the deviation from the stationary $\pi^{(clock)}$ of $P^{(clock)}$, satisfying $\sum_{i=0}^{D_\G+1} q_i = 0$.
Then one gets
$$\sum_i \vert q_i(t+D_\G+1) \vert \leq 2 (D_\G+1) \gamma \sum_i \vert q_i(t) \vert $$
by explicit computation of $(P^{(clock)})^{D_\G+1}$, bounding $q_i(t+D_\G+1)$ for $i=0,...,D_\G$ with the property $(1-\gamma)^k \leq 1$ for all $k$, and estimating $q_{D_\G+1}(t+D_\G+1)$ from $\sum_i q_i = 0$. Fixing any $\gamma$ such that $2 \gamma (D_\G+1) \leq \alpha < 1$ then indeed ensures a fast convergence rate $\alpha$ for all times. Regarding the other degrees of freedom of the lift, it is obvious that the distribution over the different nodes $(s,v_0,v)$ associated to a fixed clock value $s$, is exactly proportional to $A^s F \pi$ for any times larger than $D_\G+1$, and modulo proper initialization $\dist(0)=F \marg (0)$.
\vspace{2mm}


\section{Results on Conductance Bounds}\label{sec:conductance}

\subsection{Existing conductance bounds on the mixing time}

A key quantity, widely used in obtaining bounds on the mixing time \cite{aldous1987,lawler1988,mihail1989}, is the conductance of a stochastic $P$ on $\G$. For a subset $\X \subseteq \V$ let $\pi(\X) = \sum_{i \in \X} \pi_i$, where we recall that $\pi$ is the stationary distribution under $P$. The {\em conductance $\Phi(P)$ of $P$} is defined as \cite{levin2009}:
\begin{equation}\label{eq:conductance}
\Phi(P) = \min_{\X \subset \V ;  \pi(\X) \leq \tfrac{1}{2}}  \frac{\sum_{i \in \X, j \notin \X}  P_{j,i}\, \pi_i}{\pi(\X)} \, .
\end{equation}
This characterizes the minimal steady-state probability flow that is cut when separating the nodes into two disjoint sets of weight less than $\demi$. Given only a graph $\G$ and a target stationary distribution $\pi$ over $\V$, {\em the conductance $\Phi$ of $\G$ towards $\pi$} is the maximum of $\Phi(P)$ over all stochastic $P$ satisfying the locality constraints of $\G$ and whose unique stationary distribution is $\pi$. If $\pi$ is the uniform distribution, then $\Phi$ is upper bounded by the edge expansion of $\G$.

{The conductance can be used to bound how the minimum time for the convergence of a mixing process is constrained.}
Loosely speaking, it is known that $\tau(1/4)$ is of the order of $1/\Phi$ or larger, for any Markov chain $P$ ({\em Conductance bound}, \cite{levin2009}); and \cite{chen1999} among others proved that the same bound holds for lifted Markov chains. These bounds are however proven only in the scenario $(\mathsf{sImre}),$ {which is a quite restrictive setting among those we consider.}
Clarifying whether such a bound, or a variation thereof, would hold for different scenarios is one of the main aims for this paper.

The following relationship holds between the conductance and the diameter (see for instance \cite{fountoulakis2007}):
\[ \Phi(P) \leq \frac{4\log(1/\pi_{\min})}{D_\G-1}, \]
where $\pi_{\min}$ is the minimum element of $\pi$ and $c$ a constant.
On the other hand, the conductance can be arbitrarily small, irrespective of the diameter, as we will see with the Barbell graph in Example 2 below.

Before going on, let us briefly comment on the conductance $\Phi$ when no $P$ is imposed. This expresses how the mixing time is constrained by the {\em graph topology} and the stationary distribution $\pi$ alone and it is natural to anticipate that this will play a role in lift scenarios with $(\sfE)$. In this context, the same caveat as after \eqref{eq:properties1} is in order: when relaxing the scenario from $(\sfe)$ to $(\sfE)$, better lifts are admitted but also possibly a more favorable conductance. In particular, a fair treatment shall compare for each graph $\G$, the fastest possible lifted Markov chain $A$ (in terms of $\tau(1/4)$) with the best possible conductance $\Phi(P)$ (over all admissible $P$), where the optimal $P$ may differ from the $\tilde{P}^{(A;\lift{\pi})}$ obtained as an induced Markov chain of the fastest lift.


\subsection{Scenarios providing advantage within the conductance bound}

{\noindent Let us start with two examples.

 \noindent \textbf{Example 1 (\emph{Diaconis lift on the cycle}, see \cite{diaconis2000}):} Consider a stochastic process on the finite cycle graph, i.e.~the graph with nodes $\V=\{1,\ldots, N\}$, and where node $k$ is connected by an edge to nodes $k+1 \mod N$ and $k-1 \mod N$. The idea is to modify the standard random walk on the graph, where from any given node one moves with equal probability to either of the two neigbors, in order to make the walker's next step depend on \emph{its last move} on the cycle $\G$. Explicitly, an extended graph is constructed by associating to each node $k \in \V$ of the original graph, two nodes $(\pm 1, k)$ which indicate if the current position $k$ has been reached from $k+1$ or $k-1$ respectively. 
The lifted node set thus becomes:
$\lift{\V} = \{ (s,k) : k=1,2,...,N \text{ and } s \in \pm1\, \}.$ According to Definition \ref{def:lift}, the allowed edge set consists of the edges $\lift{\E} = \{ ((s',k\pm 1 \mod N),(s,k)) : k=1,2,...,N \text{ and } s,s' \in \pm1\, \}.$ The lifted graph is constructed by keeping a subset only of all the in principle authorized edges, while maintaining circular symmetry.
In the following, $e_{\pm}$ denotes the column vectors $(1,0)^T$ and $(0,1)^T$, $P^{(\pm 1)}$ denote clockwise and anti-clockwise rotation on the cycle $\G$, respectively, and $\otimes$ is the Kronecker product. In a matrix representation, the lifted transition map takes the form:
$$A = \sum_{i,j \in \pm 1}  Q_{i,j} \; e_i e_j^\adj \otimes P^{(j)} \; .$$
A walker on node $(+,k)$ will thus always move towards position $k+1$, although in principle the lifted graph could also have allowed it to move back to $k-1$. The element $Q$ is a stochastic matrix which defines a Markov chain on the additional states $\pm1$, with $Q_{+1,+1}=Q_{-1,-1} = 1-1/N$ and $Q_{+1,-1}=Q_{-1,+1} = 1/N$. With this choice, maintaining the same ``direction'' of movement is preferred, but there is a small $1/N$ probability of switching. The allowed transitions and the relative probabilities are depicted in Figure \ref{fig:Diaconis}.

The mixing time $\tau(1/4)$ \emph{of the whole distribution on $\lift{\V}$} with this lift is of order $N$, while the mixing time $\tau(1/4)$ of any non-lifted walk on the cycle (see e.g.~\cite{gerencser2011}, showing the best is the simple walk with probability $1/2$ to take each edge) would be of order $N^2$.
  
\begin{figure}[htb]
	\centering
	\begin{minipage}{.49\textwidth}
     \def\svgwidth{\columnwidth}
     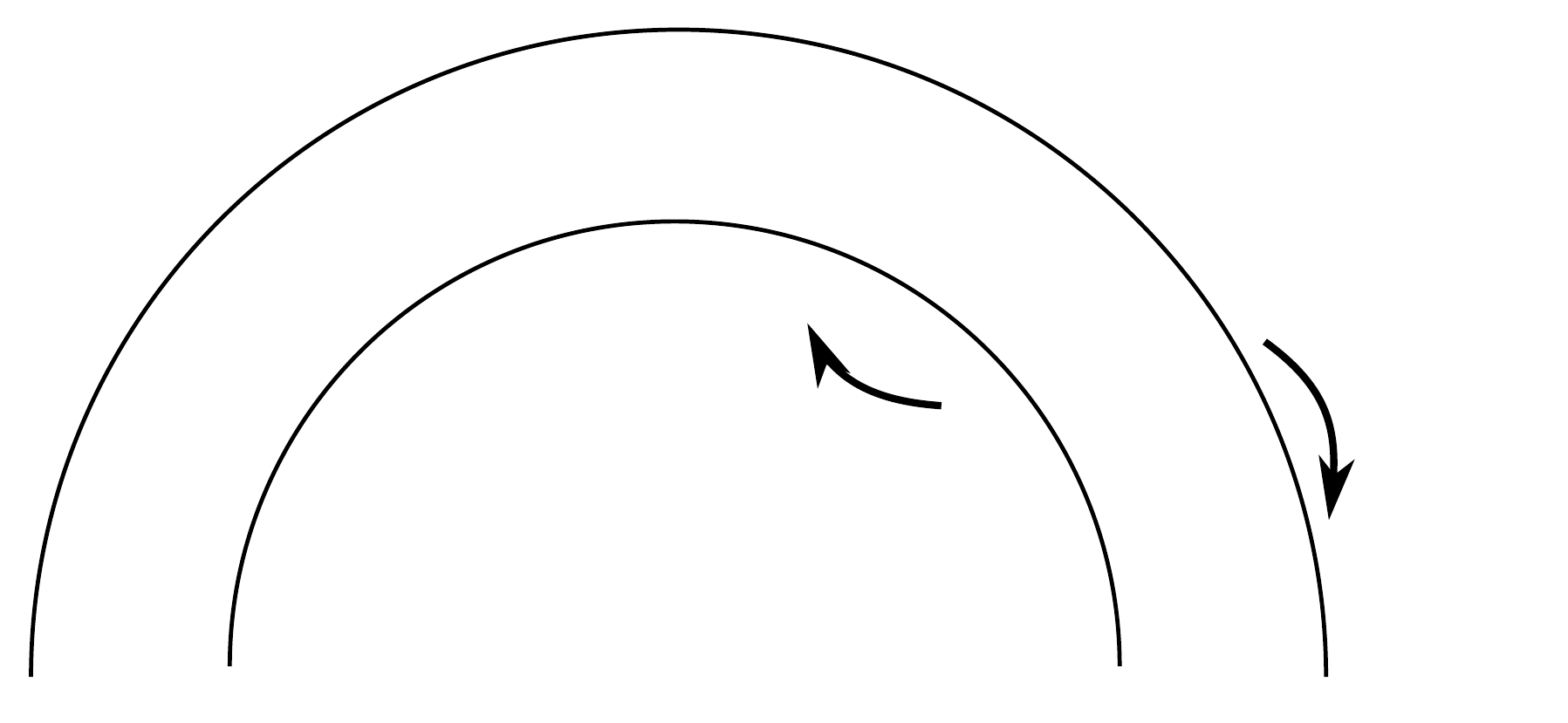
    \end{minipage} \begin{minipage}{.49\textwidth}\vspace{5mm}
    \includegraphics[width=0.95\columnwidth]{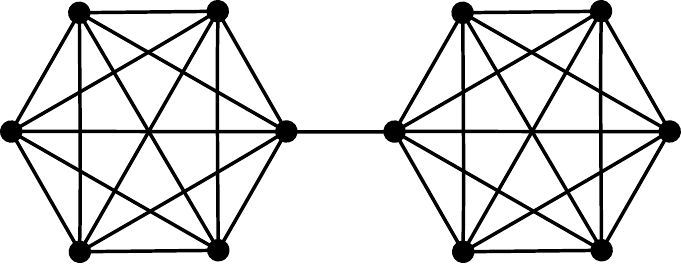}
    \end{minipage}
	\caption{(upper) The transitions for the ``Diaconis lift'' on the cycle with $N$ nodes. (lower) The Barbell graph on $2n=12$ nodes.}
	\label{fig:Diaconis}
\end{figure}

For large $N$, this speedup becomes significant. It is obtained within a $(\sfs\sfI\sfm\sfr\sfE)$ scenario, or $(\sfs\sfI\sfm\sfr\sfe)$ with the reasonable constraint of imposing ergodic flows with circular symmetry. This lift does not satisfy invariance, i.e.~starting at $x$ with $C x = \pi$ would not necessarily imply $C A x = \pi$. Indeed, just consider e.g.~$\dist_i(0)=1/N$ for $i=(+1,3)$ and $i=(-1,1)$ $\in \lift{\V}$, and an arbitrary distribution of the weight $1/N$ over the set $\frak{c}^{-1}(k)$ associated to each other node of $k \in \V\setminus\{1,3\}$. Then in the first step, no weight can flow to the nodes $\frak{c}^{-1}(2) = \{(+1, 2),(-1,2)\}$,  which in turn lose their own weight to neighbors $(\pm1,1)$ and $(\pm1,3)$, i.e.~we have $(C x)_2 = 0$. This loss of invariance is consistent with the fact that else, i.e.~with $(\sfs\sfi)$, the setting would have to satisfy Theorem \ref{thm:nolift} and hence could not feature any speedup.

The diameter of the cycle graph is of course $D_\G = N/2$, so the Diaconis lift appears to reach (order of) the optimum considered in Theorem \ref{thm:diameter}, yet by satisfying the more constraining setting $(\sfs\sfI)$. Now since for the cycle in particular the conductance is of order $1/N$, the Diaconis lift satisfies the $1/\Phi$ conductance bound as well, which here is equivalent to the diameter bound. The next example shows that for some graphs, $1/\Phi$ can significantly differ from $D_\G$. 

For completeness, let us mention that for the cycle one can build a lift satisfying $(\sfS\sfi)$ and converging \emph{exactly} to the uniform distribution in $D_\G+1$ steps. Indeed the lift constructed for the proof of Theorem 2, by exploiting the symmetry, will satisfy invariance. This does not contradict the upcoming bounds involving conductance because for the cycle, conductance and diameter give the same bound. It does warrant caution however, namely highlighting that a lower bound $\tau_*$ on $\tau(1/4)$ does not necessarily imply an associated exponential convergence with characteristic time $\tau_*$.
\hfill $\square$}
\vspace{2mm}

\noindent \textbf{Example 2 (\emph{Barbell graph}):} The $2n$-node Barbell graph $K_n-K_n$, see Figure \ref{fig:Diaconis}, consists of two completely connected graphs on $n$ nodes, connected by a single ``central'' edge $(n,n+1) \in \E$. This graph is a notable example in mixing time studies because of the clear bottleneck behavior of this central edge \cite{aldous2002,boyd2009}. As a consequence, the inverse conductance $1/\Phi$ associated to the uniform distribution $\pi$ on the Barbell graph, is unavoidably significantly larger than the diameter $D_\G = 3$. Indeed, consider the particular cut $\X=\X_n$ where $\X_n$ contains all the nodes on one side of the central edge $(n,n+1)$; the latter is thus the only one to be cut. We have
\begin{align*}\label{eq:conductance}
\Phi &= \max_P \Phi(P) = \max_P \min_{\X \subset \V ;  \pi(\X) \leq \tfrac{1}{2}}  \frac{\sum_{i \in \X, j \notin \X}  P_{j,i}\, \pi_i}{\pi(\X)} \\
& \leq \max_P  \frac{\sum_{i \in \X_n, j \notin \X_n}  P_{j,i}\, \pi_i}{\pi(\X_n)} = \max_P  \frac{P_{n,n+1}\, \tfrac{1}{2n}}{1/2} \leq 1/n
\end{align*}
since the central edge can at most have $P_{n,n+1}=1$. \hfill $\square$
\vspace{2mm}

\vspace{3mm} The previous examples show that there exist situations where (i) lifts do allow to significantly accelerate mixing, compared to the best non-lifted walks; and (ii) the standard conductance bound is significantly more constraining than the diameter bound of Theorem \ref{thm:diameter}. The aim of this section is precisely to identify scenarios where, while a lifted Markov chain could outperform the non-lifted chains (as in Example 1), it can never significantly beat the conductance bound. 
 It is important to note that, in this sense, {Example 1 is not just a particular case, and the lifts satisfying the constraints in the scenarios characterized here} are indeed expected to exceed the performance of the best possible non-lifted Markov chain $P$. In support of this, \cite{chen1999} shows that  a lift satisfying the constraints $(\sfs\sfI\sfm\sfr\sfe)$ can reach the conductance bound, modulo allowing a stochastic ``stopping rule'' (see e.g.~\cite{lovasz1998}).

\vspace{3mm}\noindent  We start with some preliminary results. 

%
%
%
%
%


\begin{lemma} \label{lemma:cut}
Consider a stochastic matrix $P$, not necessarily irreducible, on a node set $\V$ and one of its stationary distributions $\pi$. Take any $\X \subseteq \V$ such that $\pi(\X) \neq 0$ and define the distribution $\tilde{\pi}^{(\X)}$ by 
$$\tilde{\pi}^{(\X)}_i = \left\{\begin{array}{l}
\eta \, \pi_i \text{ for } i \in \X \\
0 \text{ for } i \notin \X
\end{array}\right.
\quad \text{with } \tfrac{1}{\eta} = \sum_{i \in \X} \pi_i = \pi(\X) \; .$$
Then for all $t \geq 1$ we have
$$ \sum_{j \notin \X} (P^t \tilde{\pi}^{(\X)})_j  \; \leq \; t \, \Phi_{\X,\pi}(P) \; ,$$
where $\Phi_{\X,\pi}(P) = \frac{\sum_{i \in \X, j \notin \X}  P_{j,i}\, \pi_i}{\pi(\X)}\; $, can be viewed as a conductance associated to $\pi$ and the particular subset $\X$.
\end{lemma}
\begin{proof}
We will first prove and later use the following facts:
\begin{eqnarray} \label{eq:ineq}
\sum_{j \notin \X} (P \tilde{\pi}^{(\X)})_j &=& \left\Vert P \tilde{\pi}^{(\X)} - \tilde{\pi}^{(\X)} \right\Vert_{TV} \\
\sum_{j \notin \X} (P^t \tilde{\pi}^{(\X)})_j &\leq& \left\Vert P^t \tilde{\pi}^{(\X)} - \tilde{\pi}^{(\X)} \right\Vert_{TV} \, ,
	   		\;\forall t\geq 0.\nonumber
\end{eqnarray}
To obtain the equality in \eqref{eq:ineq}, we rewrite the total variation distance:
$$ \left\Vert P \tilde{\pi}^{(\X)} - \tilde{\pi}^{(\X)} \right\Vert_{TV}
	= \sum_{u \in \V : (P \tilde{\pi}^{(\X)})_u \geq \tilde{\pi}^{(\X)}_u} \; (P \tilde{\pi}^{(\X)})_u - \tilde{\pi}^{(\X)}_u \; .$$
We then observe that $(P \tilde{\pi}^{(\X)})_u \geq \tilde{\pi}^{(\X)}_u = 0$ trivially for all $u \notin \X$, while the following computations yield the opposite conclusion for all $u \in \X$. Indeed, by the definition of $\tilde{\pi}^{(\X)}$, we have
\begin{eqnarray*}
(P \tilde{\pi}^{(\X)})_u & = & \sum_{j\in \V}P_{u,j}\tilde{\pi}^{(\X)}_j = \sum_{j\in \X} P_{u,j}\frac{\pi_j}{\pi(\X)} \\
& \leq & \frac{\sum_{j\in \V} P_{u,j}\pi_j}{\pi(\X)} = \frac{\pi_u}{\pi(\X)} = \tilde{\pi}^{(\X)}_u \, .
\end{eqnarray*}
Thus the rewritten total variation distance reduces to $\; \sum_{j \notin \X} \big(\, (P \tilde{\pi}^{(\X)})_j -0 \, \big)$.

To obtain the inequality in \eqref{eq:ineq}, we expand the total variation distance:
   	\begin{align*} 
   	  \|P^t \tilde{\pi}^{(\X)} - \tilde{\pi}^{(\X)} \|_{TV} 
   	  	=& \frac{1}{2}\sum_{u\in \V}|(P^t\tilde{\pi}^{(\X)})_u - \tilde{\pi}^{(\X)}_u|\\
   		=& \frac{1}{2} \sum_{u\notin \X} (P^t\tilde{\pi}^{(\X)})_u \\&+ \frac{1}{2}\sum_{u\in \X}|(P^t\tilde{\pi}^{(\X)})_u - \tilde{\pi}^{(\X)}_u| \\
   		\geq& \frac{1}{2}\sum_{u\notin \X} (P^t\tilde{\pi}^{(\X)})_u \\&+ \frac{1}{2}\left\vert\sum_{u\in \X}(P^t\tilde{\pi}^{(\X)})_u - \tilde{\pi}^{(\X)}_u\right\vert \\
		=& \frac{1}{2}\sum_{u\notin \X} (P^t\tilde{\pi}^{(\X)})_u\\&  + \frac{1}{2}\left\vert \left(1-\sum_{u\notin \X} (P^t\tilde{\pi}^{(\X)})_u\right) - 1 \right\vert \\
   		&= \sum_{u\notin \X} (P^t\tilde{\pi}^{(\X)})_u \; ,
   	\end{align*}
thus proving \eqref{eq:ineq}.	

Next we now obtain (see justifications below):
\begin{align} 
 \Vert P^t&\tilde{\pi}^{(\X)}-\tilde{\pi}^{(\X)} \Vert_{TV} \nonumber\\ \nonumber
 \leq& \Vert P^t\tilde{\pi}^{(\X)}-P^{t-1}\tilde{\pi}^{(\X)} \Vert_{TV}+ \Vert P^{t-1}\tilde{\pi}^{(\X)}-P^{t-2}\tilde{\pi}^{(\X)} \nonumber \Vert_{TV} \\ &+ \dots + \Vert P\tilde{\pi}^{(\X)} - \tilde{\pi}^{(\X)} \Vert_{TV} \nonumber\\ 
 \leq&  t \; \Vert P\tilde{\pi}^{(\X)} - \tilde{\pi}^{(\X)} \Vert_{TV} = t \; \Phi_{\X,\pi}(P) \; .
\end{align}
From the first to second line we have used the triangle inequality on the $\ell_1$ norm; from second to third line, we have used recursively that any stochastic matrix $P$ contracts the $\ell_1$ norm \cite{levin2009}, i.e.~for arbitrary distributions $\marg ^{(1)}$ and $\marg ^{(2)}$ we have $\; \Vert P \marg ^{(1)} - P \marg ^{(2)} \Vert_{TV} \leq \Vert \marg ^{(1)} - \marg ^{(2)} \Vert_{TV} \;$; for the last equality, we have used the equality in \eqref{eq:ineq} and the explicit computation
$ \; \sum_{j \notin \X} (P \tilde{\pi}^{(\X)})_j = \sum_{u\in X,j\notin X}P_{j,u}\tilde{\pi}^{(\X)}_u
= \sum_{u\in X,j\notin X} \frac{P_{j,u} \pi_u}{\pi(\X)} \; $ where the last expression is the definition of $\Phi_{\X,\pi}(P)$.
The stated result then follows by using inequality \eqref{eq:ineq} in the first line of the above chain of inequalities.
\end{proof}
\vspace{2mm}

\begin{lemma} \label{lemma:lift-mixing}
Consider a lifted Markov chain $A$ on $\lift{\V}$ and one of its steady states $\lift{\pi}$. Then the mixing time $\tau_M(1/4)$ for scenarios with $(\sfs\sfM)$ 
satisfies
$\;\; \tau_M(1/4) \geq \frac{1}{4\Phi(\tilde{P}^{(A;\lift{\pi})})},\;\;$
where $\tilde{P}^{(A;\lift{\pi})}$ is the induced chain on $\V$ associated to $A$ and $\lift{\pi}$.
\end{lemma}
\begin{proof}
Take a subset $\X\subseteq \V$ such that $\pi(\X) = \lift{\pi}(\c^{-1}(\X)) \leq 1/2$ and denote $\lift{\X} = \c^{-1}(\X)$. Define $\lift{\pi}^{(\lift{\X})}$ similarly to $\tilde{\pi}^{(\lift{\X})}$ in Lemma \ref{lemma:cut}, except now we are constructing it on the lifted space $x,\lift{\pi}, \lift{\X} = \c^{-1}(X)$ instead of on $y,\pi,\X$. We then get:
\begin{eqnarray*}
\Vert C A^t \lift{\pi}^{(\lift{\X})} - C \lift{\pi} \Vert_{TV} & = & \frac{1}{2} \sum_{i\in \V}\left\vert \sum_{j\in \c^{-1}(i)}(A^t \lift{\pi}^{(\lift{\X})})_j - \lift{\pi}_j \right\vert \\
& \geq & \tfrac{1}{2} \left\vert \sum_{i\in \X} \sum_{j\in \c^{-1}(i)}(A^t \lift{\pi}^{(\lift{\X})})_j - \lift{\pi}_j \right\vert\\&& + 
\tfrac{1}{2} \left\vert \sum_{i\notin \X} \sum_{j\in \c^{-1}(i)}(A^t \lift{\pi}^{(\lift{\X})})_j - \lift{\pi}_j \right\vert \\
& =: & \Vert C'A^t \lift{\pi}^{(\lift{\X})} - C'\lift{\pi} \Vert_{TV}
\end{eqnarray*}
where we define $C'$ by $C'_{1,v}=1$ if $\c(v) \in \X$, $C'_{0,v}=1$ if $\c(v) \notin \X$, all other $C'_{i,j}=0$. That is, $C'$ projects the distribution $x$ on $\lift{\V}$ onto a distribution among the two options, $\lift{\X}$ and its complementary. With the inverse triangle inequality and using the same notation $\Phi_{\X}$ as in Lemma \ref{lemma:cut}, we further develop:
\begin{eqnarray}\label{eq:1207p2}
\Vert C'A^t \lift{\pi}^{(\lift{\X})} \text{-} C'\lift{\pi} \Vert_{TV} & \geq & \Vert C' \lift{\pi}^{(\lift{\X})} \text{-} C'\lift{\pi}^{(\lift{\X})} \Vert_{TV}\\&& - \Vert C'A^t \lift{\pi}^{(\lift{\X})} \text{-} C'\lift{\pi}^{(\lift{\X})} \Vert_{TV} \\ \nonumber
& = & (\, 1\text{-}\pi(\X) \,) \;\; - \;\; \sum_{j \notin \lift{\X}} (A^t \lift{\pi}^{(\lift{\X})})_j \,\\ \nonumber
& \geq & \tfrac{1}{2} \;\; - \;\; t \Phi_{\lift{\X},\lift{\pi}}(A) \\ &=& \tfrac{1}{2} \;\; - \;\; t \Phi_{\X}(\tilde{P}^{(A;\lift{\pi})}) \; .
\end{eqnarray}
To get the last inequality, for the first term we have used our assumption that $\pi(\X) \leq 1/2$. For the second term we have used Lemma \ref{lemma:cut}, for which the only condition was that $\lift{\pi}(\lift{\X})>0$; the latter holds since $\lift{\pi}(\lift{\X})=\pi(\X)$ and the paper throughout assumes $\pi(i) > 0$ for all $i$. The final equality is obtained by recalling that the induced chain is defined precisely such that $\tilde{P}^{(A;\lift{\pi})}_{i,j} \pi_j = \sum_{u \in \c^{-1}(i),v \in \c^{-1} j} A_{u,v} \lift{\pi}_v$; this readily implies that for the particular pre-image subsets $\lift{\X} = \c^{-1}(\X) \subset \lift{\V}$ we do have $\; \Phi_{\lift{\X},\lift{\pi}}(A) = \Phi_{\X}(\tilde{P}^{(A;\lift{\pi})}) \; .$

If the mixing time $\tau(1/4)$ is equal to $T$, then in particular the initial condition $\lift{\pi}^{(\lift{\X})}$ must converge close enough to $\pi$ within $T$ steps, i.e.~we need
$$ \Vert C A^T \lift{\pi}^{(\lift{\X})} - C \lift{\pi} \Vert_{TV} \leq \tfrac{1}{4} \; .$$
By \eqref{eq:1207p2} this requires $\; \frac{1}{2} \;\; - \;\; T \Phi_{\X}(\tilde{P}^{(A;\lift{\pi})}) \leq 1/4\;$ i.e.~$\;T \geq 1/(4 \Phi_{\X}(\tilde{P}^{(A;\lift{\pi})}))\;$. Since this is true for all $\X \subset \V$ with $\pi(\X) \leq 1/2$, it is in particular true for the $\X$ for which the minimum value of $\Phi_{\X}(\tilde{P}^{(A;\lift{\pi})})$, i.e.~the conductance $\Phi(\tilde{P}^{(A;\lift{\pi})})$, is attained.
\end{proof}
\vspace{3mm}\noindent Note that the above results hold irrespective of having invariance property $(\sfi)$ or $(\sfI)$. 
The following result extends the bound from \cite{chen1999} in the $(\sfs\sfI\sfm\sfr\sfe)$ setting to the $(\sfs)$ and $(\sfs\sfe)$ settings, i.e., allowing for reducible lifts and for situations where we only care about the mixing time $\tau_M$ of the \emph{marginal} distribution.
\begin{thm} \label{cor:conductancecor1}
{The scenarios including $(\sfs)$ satisfy a conductance bound of the form 
$\tau(1/4) \geq \tau_M(1/4)\geq 1/(4\Phi) \;$, or $\; \tau(1/4) \geq \tau_M(1/4) \geq 1/(4\Phi(P))\;$ in scenarios with $(\sfe)$ and ergodic flows specified by $P$.}
\end{thm}
\begin{proof}
Lemma \ref{lemma:lift-mixing} applies directly to $(\sfs\sfM\sfe)$, where $(\sfe)$ imposes that $\tilde{P}^{(A;\lift{\pi})}=P$ for all $(A;\lift{\pi})$ and thus $\Phi(\tilde{P}^{(A;\lift{\pi})})=\Phi(P)$. Regarding $(\sfs\sfM\sfE)$, assume that the claim would not hold i.e.~a particular $A,\lift{\pi}$ would allow $\tau_M(1/4) < 1/(4 \Phi)$ with $\Phi$ the largest $\Phi(P)$ among all admissible $P$. This particular lift satisfies, as any other, that $\;\; \tau_M(1/4) \geq \frac{1}{4\Phi(\tilde{P}^{(A;\lift{\pi})})},\;\;$ where the induced chain $\tilde{P}^{(A;\lift{\pi})}$ is by construction an admissible $P$, yielding an admissible conductance $\Phi(P)$ on $\G$. This directly gives a contradiction. {The corresponding scenarios with $(\sfm)$ instead of $(\sfM)$ follow by \eqref{eq:properties1}.}
\end{proof}
\vspace{3mm}

{Theorem \ref{cor:conductancecor1} completes the pictures for the scenarios of type $(\sfs\sfI)$, since for $(\sfs\sfi)$ we already know a stronger bound from Theorem \ref{thm:nolift}. There remains to treat the scenarios with $(\sfS\sfi)$.}

We next show that, when looking at the convergence of the marginals in absence of requirements about irreducibility and ergodic flows, as in $(\sfM\sfR\sfE)$, {one can trade the constraint $(\sfi)$ for the constraint $(\sfs)$} by paying a small price in the mixing time.
\begin{lemma} Let $A$ be a lift in $(\sfS\sfi\sfM\sfR\sfE)$ with mixing time $\tau_M(1/4)=\tau.$  Then there exists another lift in $(\sfs\sfI\sfM\sfR\sfE)$ that has $ \tau_M(1/4)= 2\tau,$ with a lifted graph whose dimension is of order $N\dim(A).$
\end{lemma}
\begin{proof}
Consider the given lift $A$ in $(\sfS\sfi\sfM\sfR\sfE)$: we can construct the \textit{periodic node-clock lift} (see Appendix \ref{ssec:tools}) that, modulo proper initialization, first follows $A$'s evolution from $t=0$ to $t=T=\tau_M(1/4)$, and then periodically repeats this evolution. The proper initialization $F$ associates the weight $\marg_i$ to the lifted node $(t,v_0,v)=(0,i,i)$. To see how this {allows to trade $(\sfi)$ for $(\sfs)$}, we examine the evolution of an initial state $(t_i,v_0,v)$ of this periodic node-clock lift (equivalent to considering a Kronecker delta distribution as initial disctribution $\dist(0)$). After $T-t_i$ steps, each such state is mapped to the set of nodes $\mathcal{F}_0 := \{(0,v,v)\}$. From there, it follows the periodic evolution generated by $A$. 
This implies that after $T=\tau_M(1/4)$ more steps, i.e.~at $t=t_* = 2 T-t_i$, for any $v_0,v$ defining the initial marginal $\dist(0)$, the distribution $\dist(t_*)$ satisfies $\Vert C \dist(t_*) - \pi \Vert_{TV} \leq 1/4$ and has support on the image of the initialization map $F$. By construction we can thus write $\dist(t_*) = F \marg (t_*)$ for some distribution $\marg(t_*)$ over $\V$. For the next $T$ steps, we thus have $\dist(t_*+t) = A^t F \marg(t_*)$. Invariance $(\sfI)$ of $\pi$ under $A$ 
 means $C A^t F \pi = C F \pi = \pi$. Therefore we can write:
\begin{eqnarray*}
\Vert C \dist(t_*+t)- \pi \Vert_{TV} & = & \Vert C A^t F \marg (t_*) - C A^t F \pi \Vert_{TV} \;\\ &\leq& \; \Vert A^t F \marg (t_*) - A^t F \pi \Vert_{TV}\\
& \leq  & \Vert F \marg (t_*) - F \pi \Vert_{TV} \\&=& \Vert \marg (t_*)-\pi \Vert_{TV} \; .
\end{eqnarray*}
The second inequality holds because the stochastic matrix $A$ contracts the $\ell_1$ norm. The last equality holds since the $F$ associated to our bridge is just a relabeling of nodes, from $\V$ to a subset of $\lift{\V}$. Since we had $\Vert C \dist(t_*) - \pi \Vert_{TV} \leq 1/4$ already, we have just shown that $\Vert C \dist(t) - \pi \Vert_{TV} \leq 1/4$ for all $t \geq t_*$, when we start with all the weight of $\dist(0)$ concentrated on a single node $\in \lift{\V}$.

For an arbitrary $\dist(0)$, the state after $2 T$ steps is a convex combination of cases with $\dist _i(0)=1$, so the property $\Vert C \dist(t) - \pi \Vert_{TV} \leq 1/4$ for all $t \geq 2 T \geq t_*$ is maintained. The periodic node-clock-lift is thus a chain in $(\sfs\sfI\sfM\sfR\sfE)$, whose $\tau_M(1/4)$ mixing time is at most twice the one of the chain $A$ in $(\sfS\sfi\sfM\sfR\sfE)$.
\end{proof}

{{\em Remark:} The lift constructed in the proof (potentially) loses invariance. This is the case because the periodic node-clock-lift is built to follow $A$, and thus leave $\pi$ invariant, only if initialized in the set $\mathcal{F}_0 := \{(0,v,v)\}$; as soon as in trading $(\sfS)$ towards $(\sfs)$ we relax this initialization and allow starting at nodes $(t_i,v_0,v)$ with $t_i\neq 0$, we have no guarantee anymore about what happens when $C x(0) = \pi$. In fact, we know from Theorem \ref{thm:nolift} that no lift satisfying $(\sfs\sfi)$ can be faster than the best non-lifted Markov chain, so for graphs where the periodic node-clock-lift construction keeps invariance, also $(\sfS\sfi\sfM\sfR\sfE)$ would not allow to go faster than the (possibly already very good) fastest non-lifted Markov chain. We also note that through the periodic node-clock-lift construction, we could lose any matching ergodic flows that were present in $A$, so scenarios with $(\sfS\sfi\sfe)$ are not covered by this result.}

\begin{thm}\label{cor:lift-mixing-time}
Settings with $\;(\sfS\sfi\sfE)\;$ satisfy $\;\tau(1/4) \geq \tau_M(1/4) \geq 1/(8\Phi)\;$
\end{thm}
\begin{proof}
Since we have shown in Theorem \ref{cor:conductancecor1} that $\tau_M(1/4) \geq 1/(4\Phi)$ for $(\sfs\sfI\sfM\sfR\sfE)$, by Lemma 3 we have $\tau_M(1/4) \geq 1/(8\Phi)$ for $(\sfS\sfi\sfM\sfR\sfE)$. The same bound then holds for all scenarios with $(\sfS\sfi\sfE)$, since they are all more constrained than $(\sfS\sfi\sfM\sfR\sfE)$ and with the same conductance.
\end{proof}

We cannot strengthen $(\sfE)$ to $(\sfe)$ directly in the above proof, since this would affect both sides of the inequality: the allowed lifts, and the allowed $P$ for computing the conductance.
We thus conclude our characterization by treating the scenarios of type $(\sfS\sfi\sfe)$ on their own.
In fact, we would argue that $(\sfS\sfi\sfe)$ is perhaps the most natural scenario: a tailored initialization is allowed, invariance of the target under such initialization is required, and matching ergodic flows are imposed; irreducibility or mixing on the full distribution may be imposed or not, depending on the application. Hence, while the following discussion may appear technical and somewhat incremental, as it hinges on details, it might actually be the most important towards applications. 

{We first illustrate with an example that acceleration via lifting is indeed possible within this scenario.\vspace{2mm}

\noindent \textbf{Example 3 (\emph{Cycle on 4 nodes}):} Consider the cycle on 4 nodes, $\V=\{0,1,2,3\}$, with imposed transition probabilities $P_{0,1} = P_{1,0} = P_{2,3} = P_{3,2} = (1-\phi)\delta$, $P_{0,3}=P_{3,0}=P_{2,1}=P_{1,2}= (1-\phi)(1-\delta)$, and $P(i,i)=\phi$ for all $i \in \V$.\footnote{The self-loops are really not essential in general, they just allow to consider a simpler graph, which without self-loops would have a periodicity problem.} The corresponding stationary $\pi$ is uniform and the conductance is $\Phi(P)=(1-\phi)\cdot\min(\delta,1-\delta)$. We will show that a lifted walk satisfying $(\sfS\sfi\sfe)$, with ergodic flows as imposed, can beat the associated conductance bound, by an arbitrarily large amount if $\delta$ is taken correspondingly close to $0$ or $1$.

To this end, we construct the following directed clock-lift on $\lift{\V} = \{(s,v) \text{ with } s \in \{0,1,2\} \text{ and } v \in \V\}$, see Fig.~\ref{fig:4cycle}.
\begin{itemize}
\item Each node $(0,v) \in \lift{\V}$ sends its weight, split evenly, to the neighbors $(1,v\pm1\text{mod}4)$, i.e.$$A_{(1,v+1\text{mod}4),(0,v)}=A_{(1,v-1\text{mod}4),(0,v)}=1/2 \; ;$$
\item Each node $(1,v)$ sends half its weight to $(2,v)$, and a quarter respectively to $(2,v\pm1\text{mod}4)$:
\begin{eqnarray*}
A_{(2,v),(1,v)} &=& 1/2 \quad \text{and}\\
 A_{(2,v+1\text{mod}4),(1,v)} & =& A_{(2,v-1\text{mod}4),(1,v)} = 1/4 \; ;
 \end{eqnarray*}
\item Each node $(2,v)$ undergoes a probabilistic superposition of two different behaviors: with probability $\gamma \ll 1$, it sends all its weight back to $(0,v),$ on the first level; or, with probability $1-\gamma$, it evolves accordingly to a Markov chain similar to $P$ on the third level of the lift, namely on the subset of nodes $\mathcal{F}_2 := \{(s,v) : s=2\}$. More precisely, we choose:
\begin{eqnarray*}
A_{(2,v),(0,v)} &=& \gamma\; \text{ for all } \;v \in \V, \\
 A_{(2,0),(2,1)} &=& A_{(2,1),(2,0)}= A_{(2,2),(2,3)} = A_{(2,3),(2,2)}  \\
& =& (1-\gamma) \epsilon, \\
 A_{(2,0),(2,3)} &=& A_{(2,3),(2,0)} = A_{(2,2),(2,1)} = A_{(2,1),(2,2)} \\
& =& (1-\gamma)(1-\epsilon),\; \text{ for some}\;\epsilon >0.
\end{eqnarray*}
\end{itemize}
Next, we will jointly tune $\gamma,\epsilon,\phi$ such that the ergodic flows of $A$ match those of $P$.
\begin{figure}[t]
\hspace{-3mm}
\includegraphics[width=1\columnwidth]{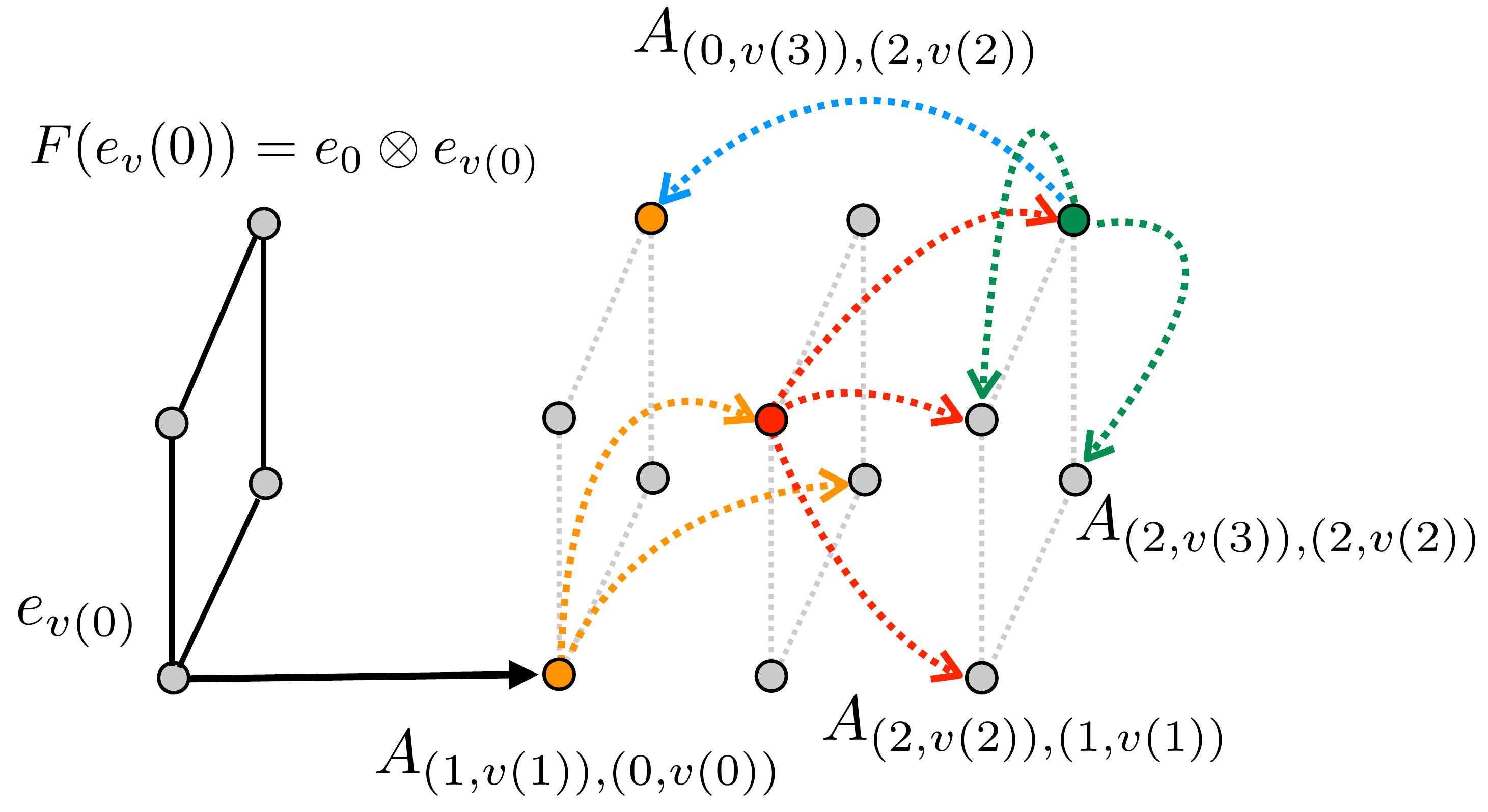}
	\caption{The node-clock lift variation for the cycle on 4 nodes, which ensures irreducibility on the lifted space. Allowed transitions for a node on each level of the lifted graph are highlighted with different colors. Notice how the last level can ``jump back".}  
\label{fig:4cycle}
\end{figure}

If we initialize the lift with $F$ in the subspace $\mathcal{F}_0 := \{(s,v) : s=0\}$, as allowed by $(\sfS)$, then $\marg=C x$ converges in two time steps exactly to $\pi$, the uniform distribution over $\V$. Moreover, property $(\sfi)$ is satisfied by symmetry and we have in fact $C \dist(t) = \pi$ for all $t\geq 2$. The lift is irreducible. To cover all scenarios in $(\sfS\sfi\sfe)$, before we check $(\sfe)$ there remains to cover $(\sfm)$, i.e.~check the convergence of $\dist(t)$ itself towards its steady state. Towards this, we note that after $t\geq 2$, i.e.~when $Cx$ has converged to $\pi$, the only remaining dynamics of $\dist(t)$ concerns the $s$ part of the lifted nodes $(s,v)$. It corresponds to a Markov chain on $\V' = \{s \in 0,1,2 \}$ and with 
$$ P' = [\; 0 \;,\; 0 \;,\; \gamma \; ; \; 1 \;,\; 0 \;,\; 0 \; ; 0 \;,\; 1 \;,\; (1-\gamma) \; ] \; .$$
This $P'$ has a steady state $[\; \frac{\gamma}{1+2\gamma} \;,\; \frac{\gamma}{1+2\gamma} \frac{1}{1+2\gamma} \;]$, and for $\gamma\ll 1$ it converges very close to this steady state in just two steps, such that $\Vert \dist(t) - \lift{\pi} \Vert_{TV} < \tfrac{1}{4}$ for all $t \geq 2$ for any $\dist(0) \in \mathcal{F}_0$. This beats $1/\Phi(P)$ by a factor $1/(2\delta)$ which can be very large for $\delta \ll 1$. 

Finally, to satisfy ($\sfe$), we must tune $\gamma,\epsilon,\phi$ to indeed make the ergodic flows of $A$ and $P$ match. From
$$\lift{\pi} = [\; \frac{\gamma}{1+2\gamma} \;,\; \frac{\gamma}{1+2\gamma} \frac{1}{1+2\gamma} \;] \otimes [\; \tfrac{1}{4} \;,\; \tfrac{1}{4} \;,\; \tfrac{1}{4} \;,\; \tfrac{1}{4} \;]$$
we directly compute $\lift{Q}^{(A;\lift{\pi})}_{\c^{-1}(i),\c^{-1}(j)}$. Identification with $Q=P/4$ yields the constraints
$$\phi = \frac{3\gamma/2}{1+2\gamma} \quad , \quad (1-2\epsilon) = \frac{1+\gamma/2}{1-\gamma}(1-2\delta)$$
for which we can clearly find satisfactory $\phi,\epsilon \in (0,1)$ for any arbitrarily small $\delta$, provided we take $\gamma$ sufficiently small such that the right hand side of the last constraint lies below $1$.

The principle of the above construction is that most of the steady state distribution is concentrated on the nodes $(s,v) \in \lift{\V}$ with $s=2$, where with high probability $(1-\gamma)$ the walk has a similarly bad behavior as $P$ itself; however when starting in the initialization set $\mathcal{F}_0$, a completely different behavior dominates and yields fast mixing on $\V$ through the first 2 time steps. Without allowing initialization in $\mathcal{F}_0$ exclusively, i.e.~in scenarios with $(\sfs)$, this advantage would drop. \hfill $\square$
\vspace{2mm}}

The above example shows that when ergodic flows are imposed, the lift can nevertheless significantly beat the conductance bound associated to these ergodic flows. However, we next show that for scenarios $(\sfS\sfi\sfe)$, the lifts can never beat the conductance bound associated to $\Phi$, the highest conductance over all possible $P$ on $\G$. In this sense, in $(\sfS\sfi\sfe)$ scenarios the lifts serve as a way to circumvent the limitations imposed by (overly-constraining) ergodic flows.

\begin{thm}\label{thm:last}
Any lifted Markov chain satisfying $(\sfS\sfi\sfe)$, has its mixing time bounded by $\tau(1/4) \geq \tau_M(1/4) \geq 1/(8\Phi)$, the conductance bound associated to the \emph{graph} $\G$.
\end{thm}
\begin{proof}
This follows from Theorem \ref{cor:lift-mixing-time}. Indeed, assume the opposite i.e.~that we can make $\tau_M(1/4)$ arbitrarily smaller than $1/(8\Phi)$ by choosing an appropriate lift. This particular lift $A$ of course remains a valid option when relaxing the ergodic flow constraint, i.e.~for the corresponding setting in $(\sfS\sfi\sfE)$. Therefore we would also have $\tau(1/4) \leq 1/(8\Phi)$ for $(\sfS\sfi\sfE)$ which is a contradiction with Theorem \ref{cor:lift-mixing-time}.
\end{proof}

The situation described by Example 3 and Theorem \ref{thm:last} is now quite clear. The more constrained scenarios $(\sfS\sfi\sfe)$ cannot go faster than those in $(\sfS\sfi\sfE)$. However, while imposing ergodic flows may directly constrain $P$ to be a slow mixer, these do not pose a hard limit on the mixing time of a lifted Markov chain when a suitable initialization is allowed. 
Example 3 above is meant to show that a speedup within this scenario is indeed possible. Its relatively simple structure suggests that there may be other and more interesting constructions with which a lift can beat the mixing time bound associated to $\Phi(P)$.
Imposing ergodic flows ($\sfe$) is rather standard in the lifted walks literature \cite{chen1999}. Our last result thus indicates that its combination with ($\sfS$) must be treated with caution. By comparing the last result with Theorem \ref{cor:conductancecor1}, it becomes apparent that trading $(\sfs)$ for $(\sfi)$ does have an effect when $(\sfe)$ is also part of the scenario. 


\section{Summary and Perspective}\label{sec:conclusion}

We have provided an extensive classification of scenarios in which lifted Markov chains can take place. It is shown that the limits and opportunities of lifts feature a subtle yet clear dependency on the scenario in which they take place. Our main occupation was to investigate the validity of conductance-like bounds on the mixing time. We have investigated 5 reasonable constraints which can be imposed on a lift, and shown that the occurrence of a conductance bound can be attributed to merely two of these. Either of the following properties of the lift imposes a conductance bound: $(\sfs)$ - impossibility of locally initializing the lift, or $(\sfi)$ - demanding invariance of the target distribution. 
We further invigorate this claim by proving that requiring {\em neither $(\sfs)$ nor $(\sfi)$} allows to mix in diameter time, so no conductance bound can be found, whereas \textit{requiring both $(\sfs)$ and $(\sfi)$} is too restrictive - a lift in this scenario can not accelerate the mixing time as compared to a non-lifted Markov chain. These results are summarized in Table \ref{tab:1}.

The main message of this analysis is that one should be careful with specifying scenarios for accelerated mixing. The relevance of imposing ergodic flows for instance, becomes rather questionable when one allows for algorithm initialization. Properly identifying such properties gains particular importance when lifted Markov chains are to be compared to other acceleration strategies, like the discrete-time quantum walks \cite{kempe2003} which we are addressing in a forthcoming paper.

Our proofs are constructive and thus also indicate the potential in the use of lifted Markov chains to speed up mixing. However, as in \cite{chen1999}, from an algorithmic viewpoint the value of the processes we construct is more existential than practical. Our proof of Theorem \ref{thm:diameter} for instance heavily builds on a non-distributed, extensive optimization of the edge weights 
and, as such, it is certainly not a viable option for e.g.~Markov chains used for Monte-Carlo sampling in large systems. In the light of this, a first and most interesting open problem regards the development of heuristics or suboptimal versions of our algorithms that would use only local information, in a way that is suitable for distributed implementation. To the best of our knowledge, except for particular graphs exhibiting  strong symmetries, a general way to build such a mixing process remains an open issue for Markov chains, consensus algorithms, and quantum walks. 
\\

\begin{table}[htp]
\begin{center}
\begin{tabular}{|c|c|}
\hline {\bf Scenarios} & {\bf Bound on mixing times} $\tau(1/4)$ and $\tau_M(1/4)$\\
\hline($\sfs\sfi$) & no advantage over non-lifted \\
\hline($\sfs$) & $\geq 1/(4\Phi)$ \\
($\sfs\sfe$) & $\geq 1/(4\Phi(P))$ \\
\hline($\sfi$) & $\geq 1/(8\Phi)$ \\
($\sfi\sfe$) & $\geq 1/(8\Phi)$ [NB: not $\Phi(P)$]\\
\hline$(\sfS\sfI)\setminus(\sfS\sfI\sfr\sfe)$ & $\leq D_\G+1$\\
($\sfS\sfI\sfr\sfe_\delta$) & $\leq D_\G+1$\\
\hline
\end{tabular}
\end{center}
\caption{}\label{tab:1}
\end{table}


\bibliographystyle{IEEEtran}
\bibliography{biblio}

\begin{thebibliography}{10}
\providecommand{\url}[1]{#1}
\csname url@samestyle\endcsname
\providecommand{\newblock}{\relax}
\providecommand{\bibinfo}[2]{#2}
\providecommand{\BIBentrySTDinterwordspacing}{\spaceskip=0pt\relax}
\providecommand{\BIBentryALTinterwordstretchfactor}{4}
\providecommand{\BIBentryALTinterwordspacing}{\spaceskip=\fontdimen2\font plus
\BIBentryALTinterwordstretchfactor\fontdimen3\font minus
  \fontdimen4\font\relax}
\providecommand{\BIBforeignlanguage}[2]{{%
\expandafter\ifx\csname l@#1\endcsname\relax
\typeout{** WARNING: IEEEtran.bst: No hyphenation pattern has been}%
\typeout{** loaded for the language `#1'. Using the pattern for}%
\typeout{** the default language instead.}%
\else
\language=\csname l@#1\endcsname
\fi
#2}}
\providecommand{\BIBdecl}{\relax}
\BIBdecl

\bibitem{dyer1991}
M.~Dyer, A.~Frieze, and R.~Kannan, ``A random polynomial-time algorithm for
  approximating the volume of convex bodies,'' \emph{Journal of the ACM
  (JACM)}, vol.~38, no.~1, pp. 1--17, 1991.

\bibitem{jerrum2004}
M.~Jerrum, A.~Sinclair, and E.~Vigoda, ``A polynomial-time approximation
  algorithm for the permanent of a matrix with nonnegative entries,''
  \emph{Journal of the ACM (JACM)}, vol.~51, no.~4, pp. 671--697, 2004.

\bibitem{kirkpatrick1983}
S.~Kirkpatrick, C.~D. Gelatt, M.~P. Vecchi \emph{et~al.}, ``Optimization by
  simulated annealing,'' \emph{science}, vol. 220, no. 4598, pp. 671--680,
  1983.

\bibitem{martinelli1999}
F.~Martinelli, ``Lectures on glauber dynamics for discrete spin models,'' in
  \emph{Lectures on probability theory and statistics}.\hskip 1em plus 0.5em
  minus 0.4em\relax Springer, 1999, pp. 93--191.

\bibitem{metropolis1953}
N.~Metropolis, A.~W. Rosenbluth, M.~N. Rosenbluth, A.~H. Teller, and E.~Teller,
  ``Equation of state calculations by fast computing machines,'' \emph{The
  journal of chemical physics}, vol.~21, no.~6, pp. 1087--1092, 1953.

\bibitem{hastings1970}
W.~K. Hastings, ``Monte carlo sampling methods using markov chains and their
  applications,'' \emph{Biometrika}, vol.~57, no.~1, pp. 97--109, 1970.

\bibitem{aldous2002}
D.~Aldous and J.~Fill, ``Reversible markov chains and random walks on graphs,''
  2002.

\bibitem{levin2009}
D.~A. Levin, Y.~Peres, and E.~L. Wilmer, \emph{Markov chains and mixing
  times}.\hskip 1em plus 0.5em minus 0.4em\relax American Mathematical Soc.,
  2009.

\bibitem{lawler1988}
G.~F. Lawler and A.~D. Sokal, ``Bounds on the 𝐿$^2$ spectrum for markov
  chains and markov processes: a generalization of cheeger’s inequality,''
  \emph{Transactions of the American mathematical society}, vol. 309, no.~2,
  pp. 557--580, 1988.

\bibitem{diaconis1991}
P.~Diaconis and D.~Stroock, ``Geometric bounds for eigenvalues of markov
  chains,'' \emph{The Annals of Applied Probability}, pp. 36--61, 1991.

\bibitem{lovasz1998}
L.~Lov{\'a}sz and P.~Winkler, ``Mixing times,'' \emph{Microsurveys in discrete
  probability}, vol.~41, pp. 85--134, 1998.

\bibitem{diaconis2000}
P.~Diaconis, S.~Holmes, and R.~M. Neal, ``Analysis of a nonreversible markov
  chain sampler,'' \emph{Annals of Applied Probability}, pp. 726--752, 2000.

\bibitem{chen1999}
F.~Chen, L.~Lov{\'a}sz, and I.~Pak, ``Lifting markov chains to speed up
  mixing,'' in \emph{Proceedings of the thirty-first annual ACM symposium on
  Theory of computing}.\hskip 1em plus 0.5em minus 0.4em\relax ACM, 1999, pp.
  275--281.

\bibitem{diaconis2013}
P.~Diaconis and L.~Miclo, ``On the spectral analysis of second-order markov
  chains,'' \emph{Ann. Fac. Sci. Toulouse Math.}, vol.~22, no.~3, pp. 573--621,
  2013.

\bibitem{kempe2003}
J.~Kempe, ``Quantum random walks: an introductory overview,''
  \emph{Contemporary Physics}, vol.~44, no.~4, pp. 307--327, 2003.

\bibitem{turitsyn2011}
K.~S. Turitsyn, M.~Chertkov, and M.~Vucelja, ``Irreversible monte carlo
  algorithms for efficient sampling,'' \emph{Physica D: Nonlinear Phenomena},
  vol. 240, no.~4, pp. 410--414, 2011.

\bibitem{rey2016}
L.~Rey-Bellet and K.~Spiliopoulos, ``Improving the convergence of reversible
  samplers,'' \emph{Journal of Statistical Physics}, vol. 164, no.~3, pp.
  472--494, 2016.

\bibitem{ramanan2016}
K.~Ramanan and A.~Smith, ``Bounds on lifting continuous markov chains to speed
  up mixing,'' \emph{arXiv preprint arXiv:1606.03161}, 2016.

\bibitem{bierkens2016}
J.~Bierkens, ``Non-reversible metropolis-hastings,'' \emph{Statistics and
  Computing}, vol.~26, no.~6, pp. 1213--1228, 2016.

\bibitem{pavon2010}
M.~Pavon and F.~Ticozzi, ``Discrete-time classical and quantum markovian
  evolutions: Maximum entropy problems on path space,'' \emph{Journal of
  Mathematical Physics}, vol.~51, no.~4, p. 042104, 2010.

\bibitem{georgiou2015}
T.~T. Georgiou and M.~Pavon, ``Positive contraction mappings for classical and
  quantum schr{\"o}dinger systems,'' \emph{Journal of Mathematical Physics},
  vol.~56, no.~3, p. 033301, 2015.

\bibitem{rabiner1986}
L.~Rabiner and B.~Juang, ``An introduction to hidden markov models,''
  \emph{ieee assp magazine}, vol.~3, no.~1, pp. 4--16, 1986.

\bibitem{motwani2010}
R.~Motwani and P.~Raghavan, \emph{Randomized algorithms}.\hskip 1em plus 0.5em
  minus 0.4em\relax Chapman \& Hall/CRC, 2010.

\bibitem{boyd2004}
S.~Boyd, P.~Diaconis, and L.~Xiao, ``Fastest mixing markov chain on a graph,''
  \emph{SIAM review}, vol.~46, no.~4, pp. 667--689, 2004.

\bibitem{jung2010}
K.~Jung, D.~Shah, and J.~Shin, ``Distributed averaging via lifted markov
  chains,'' \emph{IEEE Transactions on Information Theory}, vol.~56, no.~1, pp.
  634--647, 2010.

\bibitem{hendrickx2014}
J.~M. Hendrickx, R.~M. Jungers, A.~Olshevsky, and G.~Vankeerberghen, ``Graph
  diameter, eigenvalues, and minimum-time consensus,'' \emph{Automatica},
  vol.~50, no.~2, pp. 635--640, 2014.

\bibitem{aldous1987}
D.~Aldous, ``On the markov chain simulation method for uniform combinatorial
  distributions and simulated annealing,'' \emph{Probability in the Engineering
  and Informational Sciences}, vol.~1, no.~01, pp. 33--46, 1987.

\bibitem{mihail1989}
M.~Mihail, ``Conductance and convergence of markov chains-a combinatorial
  treatment of expanders,'' in \emph{Foundations of computer science, 1989.,
  30th annual symposium on}.\hskip 1em plus 0.5em minus 0.4em\relax IEEE, 1989,
  pp. 526--531.

\bibitem{fountoulakis2007}
N.~Fountoulakis and B.~A. Reed, ``Faster mixing and small bottlenecks,''
  \emph{Probability Theory and Related Fields}, vol. 137, no.~3, pp. 475--486,
  2007.

\bibitem{gerencser2011}
B.~Gerencs{\'e}r, ``Markov chain mixing time on cycles,'' \emph{Stochastic
  Processes and their Applications}, vol. 121, no.~11, pp. 2553--2570, 2011.

\bibitem{boyd2009}
S.~Boyd, P.~Diaconis, P.~Parrilo, and L.~Xiao, ``Fastest mixing markov chain on
  graphs with symmetries,'' \emph{SIAM Journal on Optimization}, vol.~20,
  no.~2, pp. 792--819, 2009.

\bibitem{horn2012}
R.~A. Horn and C.~R. Johnson, \emph{Matrix analysis}.\hskip 1em plus 0.5em
  minus 0.4em\relax Cambridge university press, 2012.

\end{thebibliography}


\appendix
\subsection{Linear representation of Markov process on graphs and their products}\label{app:linear}
In describing Markov chains on graphs, building their lifts and analyzing their performance by resorting to system-theoretic ideas, it is particularly convenient to use a matrix representation. Consider a graph $\G=(\V,\E)$, with $\V=\{1,\ldots,N\}$ and the associated real vector space $\RR^N.$
We will use the notation $e_i$ to canonical basis vectors of $\RR^N,$ namely vectors with all elements zero except its $i$'th element equal to one; this denotes a probability distribution on $\G$ with all weight on node $i \in \V$. The notation $e_i$ will be used more generally for elementary vectors whose dimension is clear by the context.

A real function $f$ on $\V$ can then be associated to a vector, which we still denote $f$ with some abuse of notation, $f=\sum_if(i)e_i.$ The value of the function on $i$ is then obtained as $\langle e_i,f\rangle=e_i^\adj f,$ where $\adj$ indicates the transpose (adjoint) of a vector or matrix.
In particular, we associate in this way vectors to probability distributions on $\V.$ The set of the resulting vectors, which we call probability vectors, is denoted by $\PP\siggy{N}$ and its elements have non-negative entries which sum to one.

A Markov chain on the graph is a Markov discrete-time stochastic process $\{v(t)\}_{t\geq 0}$ on the node space, with conditional probabilities $\text{Proba}\big(v(t+1)=i|v(t)=j\big)=P_{i,j}$, and $P_{i,j}\neq 0$ only if $(i,j)$ is an edge of the graph i.e.~$(i,j) \in \E$.
If $\marg(t)$ is the probability vector associated to the distribution of state of the Markov chain at time $t,$ then its evolution is generated by its one-step transition matrix $P=(P_{i,j}),$ via:
\[p(t+1)=P\, p(t).\]
In order for $p(t+1)$ to be a probability vector we need $P$ to be a column-stochastic matrix, i.e. $\sum_j P_{i,j}=1$ for all $i$. Notice that in probability theory, what we call $P$ is often called $P^\adj$. Our convention avoids cumbersome notations in the definitions of the lifted chains and when we take powers of the transition maps.

From $\G_1=(\V_1,\E_1)$ and $\G_2=(\V_2,\E_2)$ two graphs with node sets of cardinality $N_1$ and $N_2,$ respectively, we can construct a graph on the cartesian product $\V=\V_1\times \V_2,$ whose nodes are pairs $(i,j),$ with $i\in\V_1$ and $j\in\V_2.$  The edges $\E$ will be all the quadruples $((i,j),(k,\ell))$ such that $(i,k)\in\E_1$ and $(j,\ell)\in\E_2.$
If $e_{1,i},e_{2,j}$ are the elementary vectors associated to $i\in\V_1$ and $j\in\V_2$, respectively, we must associate to the corresponding product node $(i,j)$ the Kronecker product  vector $e_{1,i}\otimes e_{2,j}$ (see e.g. \cite{horn2012} for more details). The latter form a basis for the real space associated to $\V,$ that is $\RR^{N_1N_2}=\RR^{N_1}\otimes \RR^{N_2}.$ A time-homogeneous Markov chain on the product graph is defined by specifying a $N_1N_2\times N_1N_2$ stochastic matrix $A$, with nonzero elements only in positions corresponding to pairs of nodes connected by an edge.

\subsection{Basic lift constructions}\label{ssec:tools}

\noindent The following basic lift constructions are based on product graphs and are used in the proof of the main results. 

\vspace{3mm}\subsubsection{Clock lift} The following construction, which we call a \textit{clock-lift}, allows us to construct a \emph{time-homogeneous} lifted chain, whose marginal follows the evolution of some specified \emph{time-inhomogeneous} Markov chain represented by a finite sequence of $T$ stochastic matrices $P(0),\ldots,P(T-1)$. This is attained by including the time variable in the node space, in a way that is reminiscent of the inclusion of time as state variable in a dynamical system, in order to make it time-invariant.

Explicitly, consider the product graph of the original graph $\G=(\V,\E)$ with the path graph associated to the time interval $[0,T].$ The latter thus has node set $\{0,1,...,T\}$ and edges $(t,t+1)$ for $t=0,1,...,T-1$. The product graph produces a lift that effectively introduces $T$ copies of each original node, indexed by time. 
The basic idea is depicted in Figure \ref{fig:clocklift} for $G$ a path graph of 4 nodes. As a lifted node space we thus consider $\lift{\V} = \{ (t,v) : t \in \{0,1,...,T\} \text{ and } v \in \V \}$, and as edges the ones of the product graph. If we consider the surjective map $\c:\lift{\V}\to\V$ defined as $\c((t,v))=v$, it is easy to see that the product graph is a valid lift of $\G.$ We then construct a lifted Markov chain on $\lift \G$ by choosing a $\lift\G$-local, time-inhomogeneous stochastic matrix on $\RR^{T+1}\otimes\RR^N$ of the form:
 	\[ A = \sum_{t=1}^T e_te_{t-1}^\adj \otimes P(t) + e_Te_T^\adj \otimes I_\V, \]
where $e_t$ are the elementary vectors in $\RR^{T+1}$; $v^\adj$ denotes the adjoint (row vector) of $v$ (column vector); $\otimes$ denotes the Kronecker product as above, 
 such that $B \otimes P$ acts on vectors of $\RR^{(T+1)|\V|}= \RR^{\lift{\V}}$ when $B$ acts on $\RR^{(T+1)}$ and $P$ on $\RR^{|\V|}$; and $I_\V$ is the identity on $\RR^{|\V|}$. This $A$ produces the target time-inhomogeneous evolution of the marginals, when it is initialized (as in the description of the ($\sfS$) scenarios) using the linear map :
$$F : \marg (0) \mapsto \dist(0) = e_0\otimes \marg (0) \, .$$
The evolution of the lifted distribution then follows
$$ \dist(t) = A^t\dist(0) = e_t\otimes P(t)P(t-1)\dots P(1)\marg (0) \, , $$
so that $C\dist(t) = P(t)P(t-1)\dots P(1)\marg (0) = \marg (t)$ for all $0\leq t\leq T$ and $C\dist(t) = \marg (T)$ for all $t\geq T$.\vspace{2mm}

A \textit{periodic clock-lift} is a variant of the clock-lift where the product graph is constructed by replacing the path in time by a cycle, obtained by connecting the time-index $e_{T-1}$ back to $e_0$. The corresponding evolution is now:
 	\[ A = \sum_{t=1}^{T-1} e_te_{t-1}^\adj \otimes P(t) \;+\; e_0e_{T-1}^\adj \otimes P(T) \, . \]
For an initial state $\dist(0)=e_0\otimes \marg (0)$, the output $\marg (t)$ is then given by periodically applying the time-varying Markov chain transitions $\{P(t)\}$. This makes the evolution, when it starts at an arbitrary lifted node $(t,v) \in \lift{\V}$ (as must be allowed in scenarios with $(\sfs)$), to undergo  $T-t$  steps towards some $C \dist(T-t) =: \marg'(0)$, and after that it follows the same behavior as with the well-understood initialization $x'(0) = e_0 \otimes \marg'(0)$.
Moreover, unlike the clock-lift, the periodic clock-lift can be irreducible.
At the same time, it implies that the sequence $\,P(T)P(T-1)\dots P(1) \,$ is applied periodically, which might imply more complicated dynamics than the simpler non-periodic clock-lift.

\begin{figure}[t]
\hspace{-3mm}
\includegraphics[width=1\columnwidth]
{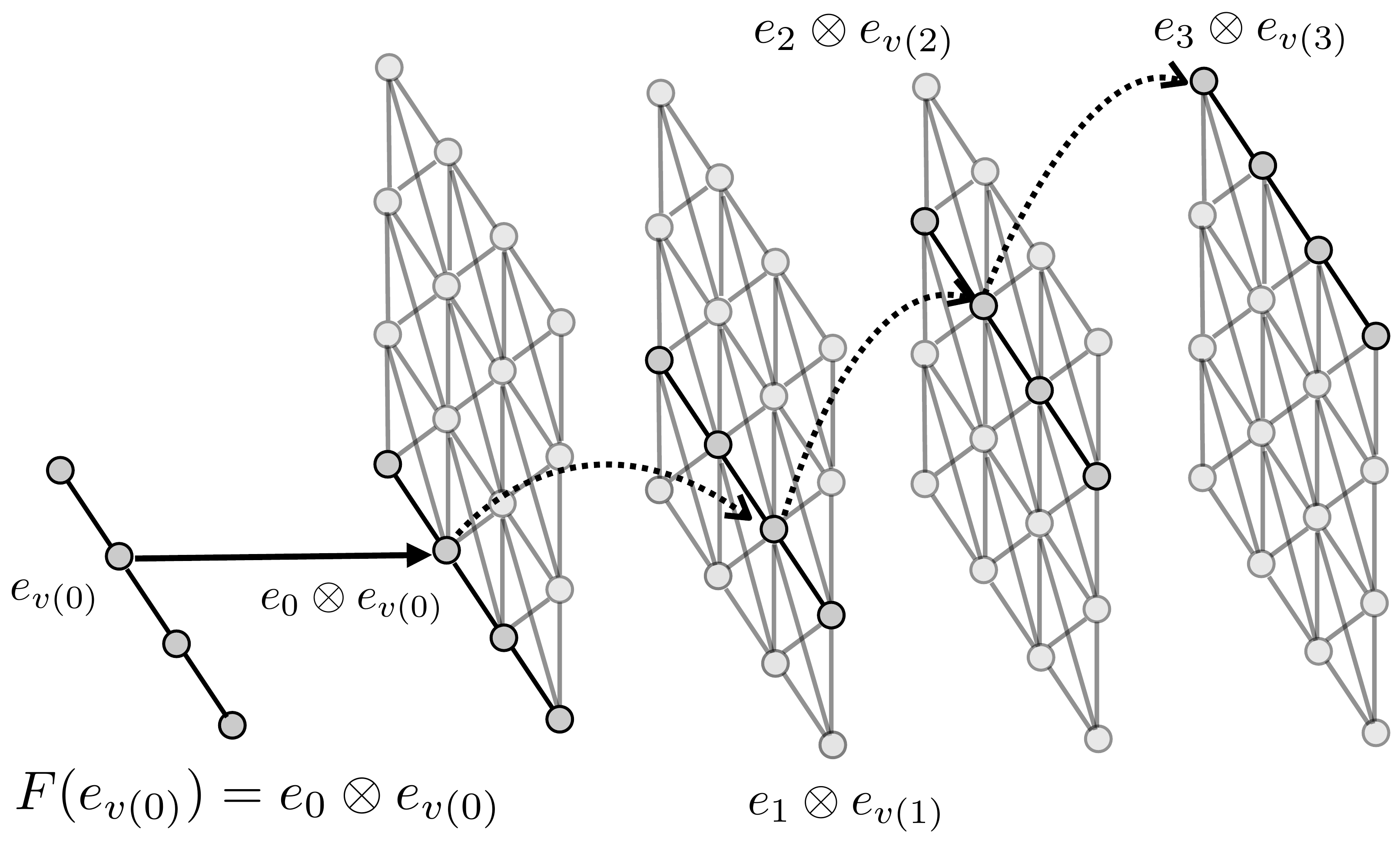}
	\caption{The initial graph $\G$ (here a path on 4 nodes) and its ``clock-lift'' enabling to enact any prescribed evolution $\marg (t) = P(t)P(t-1)\dots P(1)\marg (0)$ for all $t\leq T$, provided $\dist(0)$ is initialized with the prescribed linear $F$.}  
	\label{fig:clocklift}
\end{figure}



\vspace{3mm}\subsubsection{Node-clock lift} 
Assume that for each extreme initial state $\marg (0) = e_i$, we have built target stochastic evolutions $\marg ^{(i)}(t) = P^{(i)}(t) P^{(i)}(t-1)...P^{(i)}(1) \, e_i$, with all the sequences $\{P^{(i)}(k)\}$ satisfying the locality constraints of $\G.$ The motivation for this assumption is that such sequences can be built to attain the desired target distribution $\pi$ in minimal finite time, independently for each known initial state $e_i$ with $i=1,2,...,N$ (see next point of this Appendix). {We would then like to exploit these $N$ independent evolutions such that, starting from an initial distribution $\marg(0),$ the system follows {the whole trajectory of the particular chain} $\marg^{(i)}(t)$ with a probability $\marg_i(0)$, and thus it converges towards $\pi$ in finite time for all the valid $x(0)$ associated to any $p(0) \in \PP\siggy{N}$.

We can indeed combine these target sequences for different $i$ independently, by using a lifted Markov chain which we call a \textit{node-clock lift}. The graph is now further augmented by constructing the product graph that includes the time index, as in the previous point of this appendix, and {\em two} copies of the original graph, as depicted on Figure \ref{fig:nclift}. The lifted node space becomes $\lift{\V} = \{(t,v_0,v) : t \in \{0,1,...,T+1\}  \text{ and both } v_0,v \in \V \}$. The surjective map that shows how this is a lift is $\c:\lift{\V}\to\V$ where $\c((t,v_0,v)) = v.$ 
We consider dynamics on the lifted graph associated to
 	\begin{eqnarray} \label{eq:node-clock} 
 		A & \hspace{-10mm}=\; \sum_{t=1}^T \sum_{i\in\V} e_te_{t-1}^\adj \otimes e_ie_i^\adj \otimes P^{(i)}(t) \\
 				&+ e_{T+1} e_T^\adj \otimes \Pi^{\V}\otimes I_{\V} + e_{T+1} e_{T+1}^\adj \otimes I_{\V}\otimes I_{\V}\nonumber 
 	\end{eqnarray}
with the same notation as in the previous paragraph. Here $\Pi^{\V}$ reinitializes all $v_0 \in \V$ to the same initial state, i.e. $\Pi^{\V}x=\pi$ for all vectors $x\in\mathbb{P}_\V.$ The lift is  initialized using:
$$F : \marg (0) \mapsto \dist(0) = e_0\otimes \sum_{v\in\V} \marg _v(0) (e_v\otimes e_v) \, .$$
In particular, for a distribution $e_i$ concentrated on any initial node $i$ of the original graph, the initial state for the lift is $\dist(0) = e_0\otimes e_i\otimes e_i$.
Note that some of the lifted nodes will never be populated (e.g.~$(0,v_0,v) \in \lift{\V}$ with $v_0\neq v$), so in fact $\lift{\V}$ can be slightly reduced at the cost of a somewhat less compact description.
If each of the $\{P^{(i)}(k)\}$, set up for starting on a node $i$, was driving the state on $\V$ towards the same target distribution $\pi$ in $T$ steps, then on the lift after $T$ steps, each sub-evolution starting from $e_0\otimes e_i\otimes e_i$ will have reached the state $e_T\otimes e_{i} \otimes \pi$; and hence by convexity for \emph{any initial distribution} on $\V$, provided the lift is accordingly initialized with $F$, the induced marginal on the original state will have reached $\pi.$
The extra step does not change the marginal and is used to ensure convergence in finite time on the whole lifted space: the choice of projecting onto $\pi$ is not serving any particular purpose.
Similarly to the periodic clock-lift, we can construct a periodic node-clock-lift, where we identify $e_{T+1}\otimes e_v\otimes e_v$ with $e_0\otimes e_{v}\otimes e_v$, i.e.~any walk that has arrived at node $v$ at time $T+1$ starts again with the sequence $P^{(v)}(t)$.

\begin{figure}[t]
\hspace{-3mm}
\includegraphics[width=\columnwidth]
{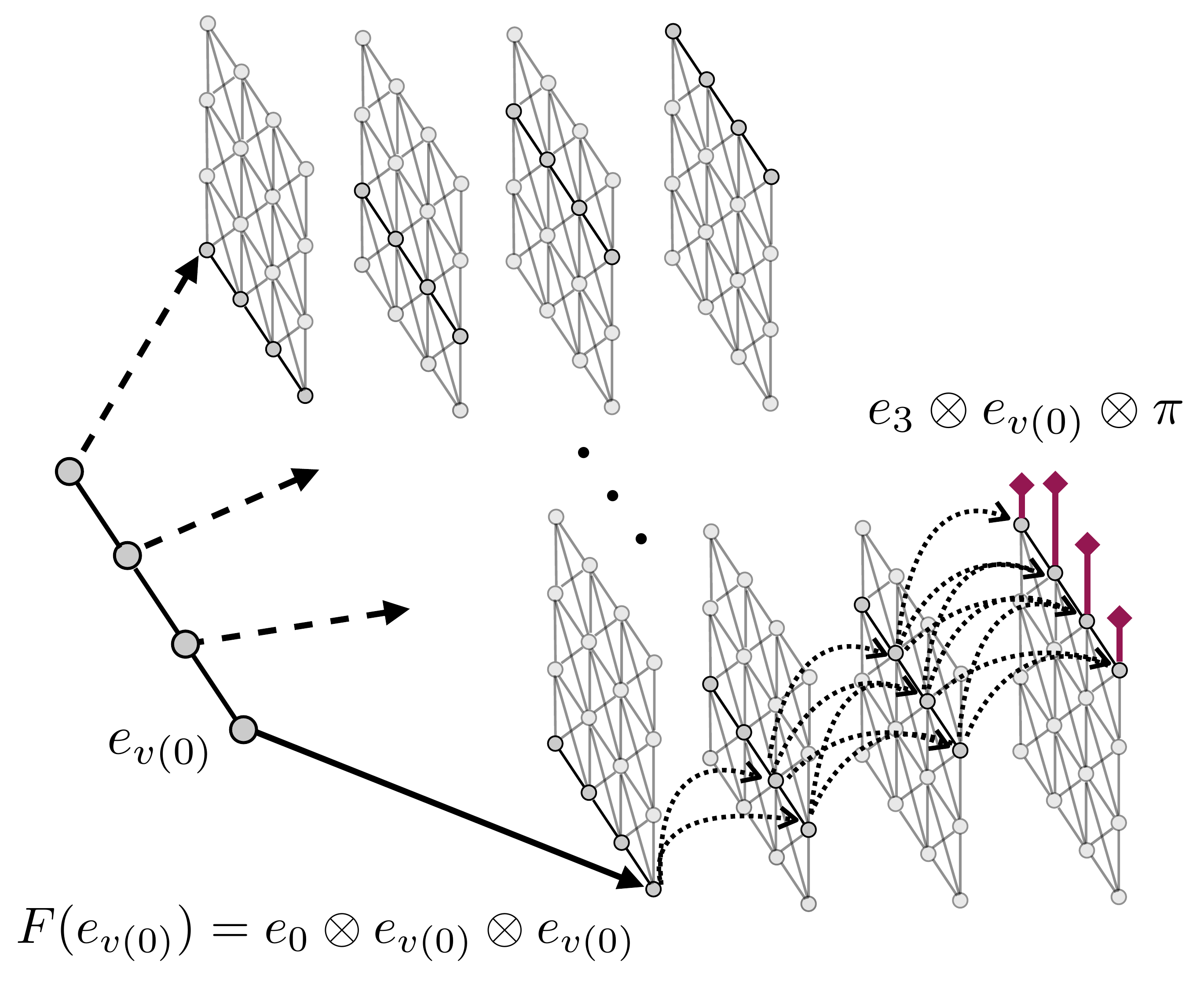}
	\caption{The same initial graph $\G$ as on Fig.\ref{fig:clocklift} (a path on 4 nodes) and
	the associated ``node-clock-lift'', in which the chosen time-inhomogeneous evolutions for each extreme initial distributions $\marg(0)=e_i$, are combined into a single time-homogeneous lifted chain by creating $N$ clock lifts.}\label{fig:nclift}
\end{figure}

\vspace{3mm}\subsubsection{Stochastic bridge}
In the above lifts, we join different $\{P^{(i)}(k)\}$ depending on the initial node $i \in \V$. In the applications of interest, these will correspond to time-inhomogeneous evolutions attaining a target distribution in finite (diameter) time.
More precisely, consider two distributions $\marg$ and $\marg'$ over the nodes $\V$ of a graph $\G$ with diameter $D_\G$. There exists a time-varying Markov chain $\{P(t)\}_{t=1}^{D_\G}$ such that $\marg'=P(D_\G)P(D_\G-1)\dots P(1)\, \marg$, where all the $P(t)$ satisfy the locality constraints imposed by $\G$, see e.g.~\cite{pavon2010,georgiou2015}. We call this sequence $\{P(t)\}_{t=1}^{D_\G}$ a \textit{stochastic bridge} from $\marg$ to $\marg'$. We shall build stochastic bridges from each $\marg=e_i$ towards the same target $\marg' = \pi$, and use the node-clock lift to join them into a single lifted Markov chain.

\end{document}